\documentclass[11pt, final]{article}
\usepackage{a4}
\usepackage{amsmath}%
\usepackage{amstext}%
\usepackage{amssymb}%
\usepackage{showkeys}%
\usepackage{epsfig}%
\usepackage{cite}
\usepackage{tikz}
\usepackage{subfig}
\usepackage{caption}
\usepackage{float}
\usepackage{multirow}
\usepackage{booktabs}

\usepackage{listings}
\newtheorem{theorem}{Theorem}

\newtheorem{algorithm}[theorem]{Algorithm}

\newtheorem{pb}[theorem]{Problem}
\newenvironment{problem}{\pb\rm}{\endpb}

\newtheorem{rem}[theorem]{Remark}
\newenvironment{remark}{\rem\rm}{\endrem}

\newcounter{unnumber}

\newtheorem{prf}[unnumber]{Proof.}
\newenvironment{proof}{\prf\rm}{\hfill{$\blacksquare$}\endprf}
\newcommand{\R}{\mathbb{R}}%
\newcommand{\N}{\mathbb{N}}%
\newcommand{\ol}{\overline}%
\newcommand{\iso}{\ensuremath{\text{iso}}}
\newcommand{\aniso}{\ensuremath{\text{aniso}}}
\newcommand{\Y}{\ensuremath{\mathcal{Y}}}
\DeclareMathOperator*\inte{int}%
\DeclareMathOperator*\sqri{sqri}%
\DeclareMathOperator*\ri{ri}%
\DeclareMathOperator*\dom{dom}%
\DeclareMathOperator*\B{\overline{\R}}%
\DeclareMathOperator*\gr{Gr}%
\DeclareMathOperator*\ran{ran}%
\DeclareMathOperator*\id{Id}%
\DeclareMathOperator*\prox{prox}%

\title{On the convergence rate improvement of a primal-dual splitting algorithm for solving monotone inclusion problems}

\author{Radu Ioan Bo\c{t} \thanks{Department of Mathematics, Chemnitz University of Technology,
D-09107 Chemnitz, Germany, e-mail:
 radu.bot@mathematik.tu-chemnitz.de. Research partially supported by DFG (German Research Foundation), project BO 2516/4-1.} \and Ern\"{o} Robert Csetnek \thanks {Department of Mathematics, Chemnitz University of Technology, D-09107 Chemnitz, Germany, e-mail:
 robert.csetnek@mathematik.tu-chemnitz.de. Research supported by DFG (German Research Foundation), project BO 2516/4-1.} \and Andr\'e Heinrich
\thanks{Department of Mathematics, Chemnitz University of Technology, D-09107 Chemnitz, Germany, e-mail:
 andre.heinrich@mathematik.tu-chemnitz.de. Research supported by the European Union, the European Social Fund (ESF) and prudsys AG in Chemnitz.}}

\begin{document}
\maketitle

\noindent \textbf{Abstract.} We present two modified versions of the primal-dual splitting algorithm relying on forward-backward splitting proposed in \cite{vu} for solving monotone inclusion problems. Under strong monotonicity assumptions for some of the operators involved we obtain for the sequences of iterates that approach the solution orders of convergence of ${\cal {O}}(\frac{1}{n})$ and ${\cal {O}}(\omega^n)$, for $\omega \in (0,1)$, respectively. The investigated primal-dual algorithms are fully decomposable, in the sense that the operators are processed individually at each iteration. We also discuss the modified algorithms in the context of convex optimization problems and present numerical experiments in image processing and support vector machines classification.\vspace{1ex}

\vspace{1ex}


\section{Introduction and preliminaries}\label{sec-intr}

The problem of finding the zeros of the sum of two (or more) maximally monotone operators in Hilbert spaces continues to be a very active research field, with applications in convex optimization, partial differential equations, signal and image processing, etc. (see \cite{bauschke-book, b-c-h, b-h, b-h2, br-combettes, combettes, combettes-pesquet, vu}). To the most prominent methods in this area belong the proximal point algorithm for finding the zeros of a maximally monotone operator (see \cite{rock-prox}) and the Douglas-Rachford splitting algorithm for finding the zeros of the sum of two maximally monotone operators (see \cite{dor}). However, also motivated by different applications, the research community was interested in considering more general problems, in which the sum of finitely many operators appear, some of them being composed with linear continuous operators \cite{bauschke-book, br-combettes, combettes}. In the last years, even more complex structures were considered, in which also parallel sums are involved, see \cite{b-h, b-h2, combettes-pesquet, vu}.

The algorithms introduced in the literature for these issues have the remarkable property that the operators involved are evaluated separately in each iteration, either by forward steps in the case of the single-valued ones (including here the linear continuous operators and their adjoints) or by backward steps for the set-valued ones, by using the corresponding resolvents. More than that they share the common feature to be of primal-dual type, meaning that they solve not only the primal inclusion problem, but also its Attouch-Th\'era-type dual. In this context we mention the primal-dual algorithms relying on Tseng's forward-backward-forward splitting method (see \cite{br-combettes, combettes-pesquet}), on the forward-backward splitting method (see \cite{vu}) and on the Douglas-Rachford splitting method  (see \cite{b-h2}). A relevant task is to adapt these iterative methods in order be able to investigate their convergence, namely, to eventually determine convergence rates for the sequences generated by the schemes in discussion. This could be important when one is interested in obtaining an optimal solution more rapidly than in their initial formulation, which furnish ``only'' the convergence statement. Accelerated versions of  the primal-dual algorithm from  \cite{combettes-pesquet} were already provided in \cite{b-h}, whereby the reported numerical experiments emphasize the advantages of the first over the original iterative scheme.

The aim of this paper is to provide modified versions of the algorithm proposed by V\~{u} in \cite{vu} for which an evaluation of their convergence behaviour is possible. By assuming that some of the operators involved are strongly monotone, we are able to obtain for the sequences of iterates orders of convergence of ${\cal {O}}(\frac{1}{n})$ and ${\cal {O}}(\omega^n)$, for $\omega \in (0,1)$, respectively.

For the readers convenience we present first some notations which are used throughout the paper (see \cite{bo-van, b-hab, bauschke-book, EkTem, simons, Zal-carte}). Let ${\cal H}$ be a real Hilbert space with \textit{inner product} $\langle\cdot,\cdot\rangle$ and associated \textit{norm} $\|\cdot\|=\sqrt{\langle \cdot,\cdot\rangle}$. The symbols $\rightharpoonup$ and $\rightarrow$ denote weak and strong convergence, respectively. When ${\cal G}$ is another Hilbert space and $K:{\cal H} \rightarrow {\cal G}$ a linear continuous operator, then the \textit{norm} of $K$ is defined as $\|K\| = \sup\{\|Kx\|: x \in {\cal H}, \|x\| \leq 1\}$, while $K^* : {\cal G} \rightarrow {\cal H}$, defined by $\langle K^*y,x\rangle = \langle y,Kx \rangle$ for all $(x,y) \in {\cal H} \times {\cal G}$, denotes the \textit{adjoint operator} of $K$.

For an arbitrary set-valued operator $A:{\cal H}\rightrightarrows {\cal H}$ we denote by $\gr A=\{(x,u)\in {\cal H}\times {\cal H}:u\in Ax\}$ its \emph{graph}, by $\dom A=\{x \in {\cal H} : Ax \neq \emptyset\}$ its \emph{domain} and by $A^{-1}:{\cal H}\rightrightarrows {\cal H}$ its \emph{inverse operator}, defined by $(u,x)\in\gr A^{-1}$ if and only if $(x,u)\in\gr A$. We say that $A$ is \emph{monotone} if $\langle x-y,u-v\rangle\geq 0$ for all $(x,u),(y,v)\in\gr A$. A monotone operator $A$ is said to be \emph{maximally monotone}, if there exists no proper monotone extension of the graph of $A$ on ${\cal H}\times {\cal H}$. The \emph{resolvent} of $A$, $J_A:{\cal H} \rightrightarrows {\cal H}$, is defined by $J_A=(\id_{{\cal H}}+A)^{-1}$, where $\id_{{\cal H}} :{\cal H} \rightarrow {\cal H}, \id_{\cal H}(x) = x$ for all $x \in {\cal H}$, is the \textit{identity operator} on ${\cal H}$. Moreover, if $A$ is maximally monotone, then $J_A:{\cal H} \rightarrow {\cal H}$ is single-valued and maximally monotone (cf. \cite[Proposition 23.7 and Corollary 23.10]{bauschke-book}). For an arbitrary $\gamma>0$ we have (see \cite[Proposition 23.2]{bauschke-book})
$$p\in J_{\gamma A}x \ \mbox{if and only if} \ (p,\gamma^{-1}(x-p))\in\gr A$$
and (see \cite[Proposition 23.18]{bauschke-book})
\begin{equation}\label{j-inv-op}
J_{\gamma A}+\gamma J_{\gamma^{-1}A^{-1}}\circ \gamma^{-1}\id\nolimits_{{\cal H}}=\id\nolimits_{{\cal H}}.
\end{equation}

Let $\gamma>0$ be arbitrary. We say that $A$ is \textit{$\gamma$-strongly monotone} if $\langle x-y,u-v\rangle\geq \gamma\|x-y\|^2$ for all $(x,u),(y,v)\in\gr A$. A single-valued operator $A:{\cal H}\rightarrow {\cal H}$ is said to be \textit{$\gamma$-cocoercive} if $\langle x-y,Ax-Ay\rangle\geq \gamma\|Ax-Ay\|^2$ for all $(x,y)\in {\cal H}\times {\cal H}$. Moreover, $A$ is \textit{$\gamma$-Lipschitzian} if $\|Ax-Ay\|\leq \gamma\|x-y\|$ for all $(x,y)\in {\cal H}\times {\cal H}$. A single-valued linear operator $A:{\cal H} \rightarrow {\cal H}$ is said to be \textit{skew}, if $\langle x,Ax\rangle =0$ for all $x \in {\cal H}$. Finally, the \textit{parallel sum} of two operators $A,B:{\cal H}\rightrightarrows {\cal H}$ is defined by $A\Box B: {\cal H}\rightrightarrows {\cal H}, A\Box B=(A^{-1}+B^{-1})^{-1}$.

The following problem represents the starting point of our investigations (see \cite{vu}).

\begin{problem}\label{pr1}
Let ${\cal H}$ be a real Hilbert space, $z\in {\cal H}$, $A:{\cal H}\rightrightarrows {\cal H}$ a maximally monotone operator and $C:{\cal H}\rightarrow {\cal H}$ an $\eta$-cocoercive operator for $\eta>0$. Let $m$ be a strictly positive integer and for any $i\in\{1,...,m\}$ let ${\cal G}_i$  be a real Hilbert space, $r_i\in {\cal G}_i$, $B_i:{\cal G }_i \rightrightarrows {\cal G}_i$ a maximally monotone operator, $D_i:{\cal G}_i\rightrightarrows {\cal G}_i$ a maximally monotone and $\nu_i$-strongly monotone operator for $\nu_i>0$ and $L_i:{\cal H}\rightarrow$ ${\cal G}_i$ a nonzero linear continuous operator. The problem is to solve the primal inclusion
\begin{equation}\label{sum-k-primal-C-D}
\mbox{find } \ol x \in {\cal H} \ \mbox{such that} \ z\in A\ol x+ \sum_{i=1}^{m}L_i^*\big((B_i\Box D_i)(L_i \ol x-r_i)\big)+C \ol x,
\end{equation}
together with the dual inclusion
\begin{equation}\label{sum-k-dual-C-D}
\mbox{ find } \ol v_1 \in {\cal G}_1,...,\ol v_m \in {\cal G}_m \ \mbox{such that } \exists x\in {\cal H}: \
\left\{
\begin{array}{ll}
z-\sum_{i=1}^{m}L_i^*\ol v_i\in Ax+Cx\\
\ol v_i\in (B_i\Box D_i)(L_ix-r_i), i=1,...,m.
\end{array}\right.
\end{equation}
\end{problem}

We say that $(\ol x, \ol v_1,...,\ol v_m)\in{\cal H} \times$ ${\cal G}_1 \times...\times {\cal G}_m$ is a primal-dual solution to Problem \ref{pr1}, if

\begin{equation}\label{prim-dual-C-D}z-\sum_{i=1}^{m}L_i^*\ol v_i\in A\ol x+C\ol x \mbox{ and }\ol v_i\in (B_i\Box D_i)(L_i\ol x-r_i), i=1,...,m.\end{equation}

If $\ol x\in{\cal H}$ is a solution to \eqref{sum-k-primal-C-D}, then there exists $(\ol v_1,...,\ol v_m)\in$ ${\cal G}_1 \times...\times {\cal G}_m$ such that $(\ol x, \ol v_1,...,\ol v_m)$ is a primal-dual solution to Problem \ref{pr1} and, if $(\ol v_1,...,\ol v_m)\in$ ${\cal G}_1 \times...\times {\cal G}_m$ is a solution to \eqref{sum-k-dual-C-D}, than there exists $\ol x\in{\cal H}$ such that $(\ol x, \ol v_1,...,\ol v_m)$ is a primal-dual solution to Problem \ref{pr1}. Moreover, if $(\ol x, \ol v_1,...,\ol v_m)\in{\cal H} \times$ ${\cal G}_1 \times...\times {\cal G}_m$ is a primal-dual solution to Problem \ref{pr1}, then $\ol x$ is a solution to \eqref{sum-k-primal-C-D} and $(\ol v_1,...,\ol v_m)\in$ ${\cal G}_1 \times...\times {\cal G}_m$ is a solution to \eqref{sum-k-dual-C-D}.

By employing the classical forward-backward algorithm (see \cite{combettes, tseng}) in a renormed product space, V\~{u} proposed in \cite{vu} an iterative scheme for solving a slightly modified version of Problem \ref{pr1} formulated in the presence of some given weights $w_i\in(0,1], i=1,...,m$, with $\sum_{i=1}^mw_i=1$ for the terms occurring in the second summand of the primal inclusion problem. The following result is an adaption of \cite[Theorem 3.1]{vu} to Problem \ref{pr1} in the error-free case and when $\lambda_n=1$ for any $n \geq 0$.

\begin{theorem}\label{vu} In Problem \ref{pr1} suppose that $$z\in\ran \left(A+\sum_{i=1}^m L_i^*\big((B_i\Box D_i)(L_i\cdot-r_i)\big)+C\right).$$

\noindent Let $\tau$ and $\sigma_i$, $i=1,...,m$, be strictly positive numbers such that $$2\cdot\min\{\tau^{-1},\sigma_1^{-1},...,\sigma_m^{-1}\}\cdot\min\{\eta,\nu_1,...,\nu_m\}\left(1-\sqrt{\tau\sum_{i=1}^{m}\sigma_i\|L_i\|^2}\right)>1.$$
Let $(x_0,v_{1,0},...,v_{m,0}) \in {\cal H} \times $ ${\cal G}_1$ $\times...\times$ ${\cal G}_m$ and for all $n\geq 0$ set:
$$
\begin{array}{lll}
x_{n+1}=J_{\tau A}\big[x_n-\tau\big(\sum_{i=1}^{m}L_i^*v_{i,n}+Cx_n-z\big)\big]\\
y_n=2x_{n+1} - x_n\\
v_{i,n+1}=J_{\sigma_i B_i^{-1}}[v_{i,n}+\sigma_i(L_iy_n-D_i^{-1}v_{i,n}-r_i)], i=1,...,m.
\end{array}$$
Then there exists a primal-dual solution $(\ol x, \ol v_1,..., \ol v_m)$ to Problem \ref{pr1} such that $x_n\rightharpoonup \ol x$ and $(v_{1,n},...,v_{m,n})\rightharpoonup(\ol v_1,...,\ol v_m)$ as $n \rightarrow +\infty$.
\end{theorem}
The structure of the paper is as follows. In the next section we propose under appropriate strong monotonicity assumptions two modified versions of the above algorithm which ensure for the sequences of iterates orders of convergence of ${\cal {O}}(\frac{1}{n})$ and ${\cal {O}}(\omega^n)$, for $\omega \in (0,1)$, respectively. In Section \ref{sec-opt-pb} we show how to particularize the general results in the context of nondifferentiable convex optimization problems, where some of the functions occurring in the objective are strongly convex. In the last section we present some numerical experiments in image denoising and support vector machines classification and emphasize also the practical advantages of the modified iterative schemes over the initial one provided in Theorem \ref{vu}.

\section{Two modified primal-dual algorithms}\label{sec2}

In this section we propose in two different settings modified versions of the algorithm in Theorem \ref{vu} and discuss the orders of convergence of the sequences of iterates generated by the new schemes.

\subsection{The case $A+C$ is strongly monotone}\label{subsec2.1}

For the beginning, we show that in case $A+C$ is strongly monotone one can guarantee an order of convergence of ${\cal {O}}(\frac{1}{n})$ for the sequence $(x_n)_{n \geq 0}$. To this end, we update in each iteration the parameters $\tau$ and $\sigma_i$, $i=1,...,m$, and use a modified formula for the sequence $(y_n)_{n \geq 0}$. Due to technical reasons, we apply this method in case $D_i^{-1}$ is equal to zero for $i=1,...,m$, that is $D_i(0)={\cal G}_i$ and $D_i(x)=\emptyset$ for $x\neq 0$. Let us notice that, by using the approach proposed in \cite[Remark 3.2]{b-h}, one can extend the statement of Theorem \ref{th-k-op-acc1} below, which is the main result of this subsection, to the primal-dual pair of monotone inclusions as stated in Problem \ref{pr1}.

More precisely, the problem we consider throughout this subsection is as follows.

\begin{problem}\label{pr2} Let ${\cal H}$ be a real Hilbert space, $z\in {\cal H}$, $A:{\cal H}\rightrightarrows {\cal H}$ a maximally monotone operator and $C:{\cal H}\rightarrow {\cal H}$ a monotone and $\eta$-Lipschitzian operator for $\eta>0$. Let $m$ be a strictly positive integer and for any $i\in\{1,...,m\}$ let ${\cal G}_i$  be a real Hilbert space, $r_i\in {\cal G}_i$, $B_i:{\cal G }_i \rightrightarrows {\cal G}_i$ a maximally monotone operator and $L_i:{\cal H}\rightarrow$ ${\cal G}_i$ a nonzero linear continuous operator. The problem is to solve the primal inclusion
\begin{equation}\label{sum-k-primal-C}
\mbox{find } \ol x \in {\cal H} \ \mbox{such that} \ z\in A\ol x+ \sum_{i=1}^{m}L_i^*(B_i(L_i \ol x-r_i))+C \ol x,
\end{equation}
together with the dual inclusion
\begin{equation}\label{sum-k-dual-C}
\mbox{ find } \ol v_1 \in {\cal G}_1,...,\ol v_m \in {\cal G}_m \ \mbox{such that } \exists x\in {\cal H}: \
\left\{
\begin{array}{ll}
z-\sum_{i=1}^{m}L_i^*\ol v_i\in Ax+Cx\\
\ol v_i\in B_i(L_ix-r_i), i=1,...,m.
\end{array}\right.
\end{equation}
\end{problem}
As for Problem \ref{pr1}, we say that $(\ol x, \ol v_1,...,\ol v_m)\in{\cal H} \times$ ${\cal G}_1 \times...\times {\cal G}_m$ is a primal-dual solution to Problem \ref{pr2}, if
\begin{equation}\label{prim-dual-sol-C}
z-\sum_{i=1}^{m}L_i^*\ol v_i\in A\ol x+C\ol x \mbox{ and }\ol v_i\in B_i(L_i\ol x-r_i), i=1,...,m.
\end{equation}

\begin{remark}\label{acc1-cocoerc-Lip}
One can notice that, in comparison to Problem \ref{pr1}, we relax in Problem \ref{pr2} the assumptions made on the operator $C$. It is obvious that, if $C$  is a $\eta$-cocoercive operator for $\eta >0$, then $C$ is monotone and $1/\eta$-Lipschitzian. Although in case $C$ is the gradient of a convex and differentiable function, due to the celebrated Baillon-Haddad Theorem (see, for instance, \cite[Corollary 8.16]{bauschke-book}), the two classes of operators coincide, in general the second one is larger. Indeed, nonzero linear, skew and Lipschitzian operators are not cocoercive. For example, when ${\cal H}$ and ${\cal G}$ are real Hilbert spaces and $L:{\cal H}\rightarrow {\cal G}$ is nonzero linear continuous, then $(x,v)\mapsto (L^*v,-Lx)$ is an operator having all these properties. This operator appears in a natural way when considering primal-dual monotone inclusion problems as done in \cite{br-combettes}.
\end{remark}

Under the assumption that $A+C$ is $\gamma$-strongly monotone for $\gamma>0$ we propose the following modification of the iterative scheme in Theorem \ref{vu}.

\begin{algorithm}\label{alg-k-op-acc1}$ $

\noindent\begin{tabular}{rl}
\verb"Initialization": & \verb"Choose" $\tau_0>0, \sigma_{i,0}>0$, $i=1,..,m$, \verb"such that"\\
                       & $\tau_0<2\gamma/\eta$, $\lambda\geq \eta+1$, $\tau_0\sum_{i=1}^m \sigma_{i,0}\|L_i\|^2\leq \sqrt{1+\tau_0(2\gamma-\eta\tau_0)/\lambda}$\\
                       & \verb"and" $(x_0,v_{1,0},...,v_{m,0}) \in {\cal H} \times $ ${\cal G}_1$ $\times...\times$ ${\cal G}_m$. \\
\verb"For" $n\geq 0$ \verb"set": & $x_{n+1}=J_{(\tau_n/\lambda)  A}\big[x_n-(\tau_n/\lambda)\big(\sum_{i=1}^{m}L_i^*v_{i,n}+Cx_n-z\big)\big]$\\
                                 & $\theta_n=1/\sqrt{1+\tau_n(2\gamma-\eta\tau_n)/\lambda}$\\
                                 & $y_n=x_{n+1}+\theta_n (x_{n+1}- x_n)$\\
                                 & $v_{i,n+1}=J_{\sigma_{i,n} B_i^{-1}}[v_{i,n}+\sigma_{i,n}(L_iy_n-r_i)]$, $i=1,...,m$\\
                                 & $\tau_{n+1}=\theta_n\tau_n$, $\sigma_{i,n+1}=\sigma_{i,n}/\theta_{n+1}$, $i=1,...,m.$
\end{tabular}
\end{algorithm}

\begin{remark} Notice that the assumption $\tau_0\sum_{i=1}^m \sigma_{i,0}\|L_i\|^2\leq \sqrt{1+\tau_0(2\gamma-\eta\tau_0)/\lambda}$ in Algorithm \ref{alg-k-op-acc1} is equivalent to $\tau_1\sum_{i=1}^m \sigma_{i,0}\|L_i\|^2\leq 1$, being fulfilled if $\tau_0>0$ is chosen such that $$\tau_0\leq\frac{\gamma/\lambda+\sqrt{\gamma^2/\lambda^2+(\sum_{i=1}^m\sigma_{i,0}\|L_i\|^2)^2+\eta/\lambda}}{(\sum_{i=1}^m\sigma_{i,0}\|L_i\|^2)^2+\eta/\lambda}.$$
\end{remark}

\begin{theorem}\label{th-k-op-acc1} Suppose that $A+C$ is $\gamma$-strongly monotone for $\gamma>0$ and let $(\ol x,\ol v_1,...,\ol v_m)$ be a primal-dual solution to
Problem \ref{pr2}. Then the sequences generated by Algorithm \ref{alg-k-op-acc1} fulfill for any $n\geq 0$
\begin{eqnarray*}
& & \frac{\lambda\|x_{n+1}-\ol x\|^2}{\tau_{n+1}^2}+\left(1-\tau_1\sum_{i=1}^m \sigma_{i,0}\|L_i\|^2\right)\sum_{i=1}^{m}\frac{\|v_{i,n}-\ol{v}_i\|^2}{\tau_1\sigma_{i,0}} \leq\\
& & \frac{\lambda\|x_1-\ol x\|^2}{\tau_1^2}+\sum_{i=1}^{m}\frac{\|v_{i,0}-\ol v_i\|^2}{\tau_1\sigma_{i,0}} +\frac{\|x_1-x_0\|^2}{\tau_0^2} + \frac{2}{\tau_0}\sum_{i=1}^{m}\langle L_i(x_1-x_0),v_{i,0}-\ol v_i\rangle.\end{eqnarray*}
Moreover, $\lim\limits_{n\rightarrow+\infty}n\tau_n=\frac{\lambda}{\gamma}$, hence one obtains for $(x_n)_{n \geq 0}$ an order of convergence of ${\cal {O}}(\frac{1}{n})$.
\end{theorem}

\begin{proof} The idea of the proof relies on showing that the following Fej\'{e}r-type inequality is true for any $n \geq 0$
\begin{eqnarray}\label{A-B-acc1''}
& & \frac{\lambda}{\tau_{n+2}^2}\|x_{n+2}-\ol x\|^2+\sum_{i=1}^{m}\frac{\|v_{i,n+1}-\ol v_i\|^2}{\tau_1\sigma_{i,0}} + \frac{\|x_{n+2}-x_{n+1}\|^2}{\tau_{n+1}^2} - \nonumber\\
& &  \frac{2}{\tau_{n+1}}\sum_{i=1}^{m} \langle L_i(x_{n+2}-x_{n+1}),-v_{i,n+1}+\ol v_i\rangle \leq \\
& & \frac{\lambda}{\tau_{n+1}^2}\|x_{n+1}-\ol x\|^2+\sum_{i=1}^{m}\frac{\|v_{i,n}-\ol v_i\|^2}{\tau_1\sigma_{i,0}} +\frac{\|x_{n+1}-x_n\|^2}{\tau_n^2} - \nonumber \\
& & \frac{2}{\tau_n} \sum_{i=1}^{m} \langle L_i(x_{n+1}-x_n),-v_{i,n}+\ol v_i\rangle\nonumber.
\end{eqnarray}

To this end we use first that in the light of the definition of the resolvents it holds for any $n \geq 0$
\begin{equation}\label{J-A-acc1}
\frac{\lambda}{\tau_{n+1}}(x_{n+1}-x_{n+2})-\left(\sum_{i=1}^{m}L_i^*v_{i,n+1}+Cx_{n+1}-z\right)+Cx_{n+2}\in (A+C)x_{n+2}.
\end{equation}

Since $A+C$ is $\gamma$-strongly monotone, \eqref{prim-dual-sol-C} and \eqref{J-A-acc1} yield for any $n \geq 0$
\begin{eqnarray*}
& & \gamma\|x_{n+2}-\ol x\|^2\leq  \left\langle x_{n+2}-\ol x, \frac{\lambda}{\tau_{n+1}}(x_{n+1}-x_{n+2}) \right \rangle + \\
& & \left\langle x_{n+2}-\ol x, -\left(\sum_{i=1}^{m}L_i^*v_{i,n+1}+Cx_{n+1}-z\right)+Cx_{n+2}-\left(z-\sum_{i=1}^{m}L_i^*\ol v_i\right)\right\rangle = \\
& & \frac{\lambda}{\tau_{n+1}}\left\langle x_{n+2}-\ol x,x_{n+1}-x_{n+2}\right\rangle+\langle x_{n+2}-\ol x,Cx_{n+2}-Cx_{n+1}\rangle+\\
& & \sum_{i=1}^{m}\left\langle L_i(x_{n+2}-\ol x),\ol v_i-v_{i,n+1}\right\rangle.
\end{eqnarray*}

Further, we have \begin{equation}\label{hilb}\langle x_{n+2}-\ol x,x_{n+1}-x_{n+2}\rangle=\frac{\|x_{n+1}-\ol x\|^2}{2}-\frac{\|x_{n+2}-\ol x\|^2}{2}-\frac{\|x_{n+1}-x_{n+2}\|^2}{2}\end{equation}
and, since $C$ is $\eta$-Lipschitzian,
$$\langle x_{n+2}-\ol x,Cx_{n+2}-Cx_{n+1}\rangle\leq \frac{\eta\tau_{n+1}}{2}\|x_{n+2}-\ol x\|^2+\frac{\eta}{2\tau_{n+1}}\|x_{n+2}-x_{n+1}\|^2,$$
hence for any $n \geq 0$ it yields
\begin{eqnarray*}
& & \left(\frac{\lambda}{\tau_{n+1}}+2\gamma-\eta\tau_{n+1}\right)\|x_{n+2}-\ol x\|^2\leq \\
& & \frac{\lambda}{\tau_{n+1}}\|x_{n+1}-\ol x\|^2-\frac{\lambda-\eta}{\tau_{n+1}}\|x_{n+2}-x_{n+1}\|^2+2\sum_{i=1}^{m}\langle L_i(x_{n+2}-\ol x),\ol v_i-v_{i,n+1}\rangle.
\end{eqnarray*}
Taking into account that $\lambda\geq\eta+1$ we obtain for any $n \geq 0$ that
\begin{eqnarray}\label{x-n-acc1}
& & \left(\frac{\lambda}{\tau_{n+1}}+2\gamma-\eta\tau_{n+1}\right)\|x_{n+2}-\ol x\|^2\leq \nonumber \\
& & \frac{\lambda}{\tau_{n+1}}\|x_{n+1}-\ol x\|^2-\frac{1}{\tau_{n+1}}\|x_{n+2}-x_{n+1}\|^2+2\sum_{i=1}^{m}\langle L_i(x_{n+2}-\ol x),\ol v_i-v_{i,n+1}\rangle.
\end{eqnarray}

On the other hand, for every $i=1,...,m$ and any $n \geq 0$, from
\begin{equation}\label{J-B-i-acc1}
\frac{1}{\sigma_{i,n}}(v_{i,n}-v_{i,n+1})+L_iy_n-r_i\in B_i^{-1}v_{i,n+1},
\end{equation}
the monotonicity of $B_i^{-1}$ and \eqref{prim-dual-sol-C} we obtain
\begin{eqnarray*}
0 & \leq & \left\langle \frac{1}{\sigma_{i,n}}(v_{i,n}-v_{i,n+1})+L_iy_n-r_i-(L_i \ol x-r_i), v_{i,n+1}-\ol v_i\right\rangle\\
& = & \frac{1}{\sigma_{i,n}}\langle v_{i,n}-v_{i,n+1},v_{i,n+1}-\ol v_i\rangle+\langle L_i(y_n-\ol x),v_{i,n+1}-\ol v_i\rangle\\
& = & \frac{1}{2\sigma_{i,n}}\|v_{i,n}-\ol v_i\|^2-\frac{1}{2\sigma_{i,n}}\|v_{i,n}-v_{i,n+1}\|^2-\frac{1}{2\sigma_{i,n}}\|v_{i,n+1}-\ol v_i\|^2\\
& + & \langle L_i(y_n-\ol x),v_{i,n+1}-\ol v_i\rangle,
\end{eqnarray*}
hence
\begin{equation}\label{v-n-acc1}
\frac{\|v_{i,n+1}-\ol v_i\|^2}{\sigma_{i,n}}\leq \frac{\|v_{i,n}-\ol v_i\|^2}{\sigma_{i,n}}-\frac{\|v_{i,n}-v_{i,n+1}\|^2}{\sigma_{i,n}}+2\langle L_i(y_n-\ol x),v_{i,n+1}-\ol v_i\rangle.
\end{equation}
Summing up the inequalities in \eqref{x-n-acc1} and \eqref{v-n-acc1} we obtain for any $n\geq 0$
\begin{eqnarray}\label{A-B-acc1}
& & \left(\frac{\lambda}{\tau_{n+1}}+2\gamma-\eta\tau_{n+1}\right)\|x_{n+2}-\ol x\|^2+\sum_{i=1}^{m}\frac{\|v_{i,n+1}-\ol v_i\|^2}{\sigma_{i,n}}\leq \nonumber\\
& & \frac{\lambda}{\tau_{n+1}}\|x_{n+1}-\ol x\|^2+\sum_{i=1}^{m}\frac{\|v_{i,n}-\ol v_i\|^2}{\sigma_{i,n}} -\frac{\|x_{n+2}-x_{n+1}\|^2}{\tau_{n+1}}-\sum_{i=1}^{m}\frac{\|v_{i,n}-v_{i,n+1}\|^2}{\sigma_{i,n}}\\
& & + 2\sum_{i=1}^{m}\langle L_i(x_{n+2}-y_n),-v_{i,n+1}+\ol v_i\rangle. \nonumber
\end{eqnarray}

Further, since $y_n=x_{n+1}+\theta_n (x_{n+1}- x_n)$, for every $i=1,...,m$ and any $n \geq 0$ it holds
\begin{eqnarray*}
& & \langle L_i(x_{n+2}-y_n),-v_{i,n+1}+\ol v_i\rangle=\langle L_i\big(x_{n+2}-x_{n+1}-\theta_n(x_{n+1}-x_n)\big),-v_{i,n+1}+\ol v_i\rangle=\\
& & \langle L_i(x_{n+2}-x_{n+1}),-v_{i,n+1}+\ol v_i\rangle-\theta_n\langle L_i(x_{n+1}-x_n),-v_{i,n}+\ol v_i\rangle+ \\
& & \theta_n\langle L_i(x_{n+1}-x_n),-v_{i,n}+v_{i,n+1}\rangle \leq\\
& & \langle L_i(x_{n+2}-x_{n+1}),-v_{i,n+1}+\ol v_i\rangle-\theta_n\langle L_i(x_{n+1}-x_n),-v_{i,n}+\ol v_i\rangle+\\
& & \frac{\theta_n^2\|L_i\|^2\sigma_{i,n}}{2}\|x_{n+1}-x_n\|^2+\frac{\|v_{i,n}-v_{i,n+1}\|^2}{2\sigma_{i,n}}.
\end{eqnarray*}

By combining the last inequality with \eqref{A-B-acc1} we obtain for any $n \geq 0$
\begin{eqnarray}\label{A-B-acc1'}
& & \left(\frac{\lambda}{\tau_{n+1}}+2\gamma-\eta\tau_{n+1}\right)\|x_{n+2}-\ol x\|^2+\sum_{i=1}^{m}\frac{\|v_{i,n+1}-\ol v_i\|^2}{\sigma_{i,n}} + \frac{\|x_{n+2}-x_{n+1}\|^2}{\tau_{n+1}}-\nonumber\\
& & 2\sum_{i=1}^{m} \langle L_i(x_{n+2}-x_{n+1}),-v_{i,n+1}+\ol v_i\rangle \leq \\
& & \frac{\lambda}{\tau_{n+1}}\|x_{n+1}-\ol x\|^2+\sum_{i=1}^{m}\frac{\|v_{i,n}-\ol v_i\|^2}{\sigma_{i,n}} +\left(\sum_{i=1}^{m}\|L_i\|^2\sigma_{i,n}\right)\theta_n^2\|x_{n+1}-x_n\|^2 - \nonumber\\
& & 2\sum_{i=1}^{m}\theta_n\langle L_i(x_{n+1}-x_n),-v_{i,n}+\ol v_i\rangle \nonumber.
\end{eqnarray}

After dividing \eqref{A-B-acc1'} by $\tau_{n+1}$ and noticing that for any $n \geq 0$
$$\frac{\lambda}{\tau_{n+1}^2}+\frac{2\gamma}{\tau_{n+1}}-\eta=\frac{\lambda}{\tau_{n+2}^2},$$
$$\tau_{n+1}\sigma_{i,n}=\tau_n\sigma_{i,n-1}=...=\tau_1\sigma_{i,0}$$ and
$$\frac{\left(\sum_{i=1}^{m}\|L_i\|^2\sigma_{i,n}\right)\theta_n^2}{\tau_{n+1}}=
\frac{\tau_{n+1}\sum_{i=1}^{m}\|L_i\|^2\sigma_{i,n}}{\tau_n^2}=\frac{\tau_1\sum_{i=1}^{m}\|L_i\|^2\sigma_{i,0}}{\tau_n^2}\leq \frac{1}{\tau_n^2},$$
it follows that the Fej\'{e}r-type inequality \eqref{A-B-acc1''} is true.

Let $N \in \N, N \geq 2$. Summing up the inequality in \eqref{A-B-acc1''} from $n=0$ to $N-1$, it yields
\begin{eqnarray}\label{A-B-acc1-N}
& & \frac{\lambda}{\tau_{N+1}^2}\|x_{N+1}-\ol x\|^2+\sum_{i=1}^{m}\frac{\|v_{i,N}-\ol v_i\|^2}{\tau_1\sigma_{i,0}} + \frac{\|x_{N+1}-x_N\|^2}{\tau_N^2} \leq \nonumber\\
& & \frac{\lambda}{\tau_1^2}\|x_1-\ol x\|^2+\sum_{i=1}^{m}\frac{\|v_{i,0}-\ol v_i\|^2}{\tau_1\sigma_{i,0}} +\frac{\|x_1-x_0\|^2}{\tau_0^2} +\\
& & 2\sum_{i=1}^{m}\left(\frac{1}{\tau_N}\langle L_i(x_{N+1}-x_N),-v_{i,N}+\ol v_i\rangle-\frac{1}{\tau_0}\langle L_i(x_1-x_0),-v_{i,0}+\ol v_i\rangle\right). \nonumber
\end{eqnarray}

Further, for every $i=1, ..., m$ we use the inequality
\begin{eqnarray*}
& & \frac{2}{\tau_N}\langle L_i(x_{N+1}-x_N),-v_{i,N}+\ol v_i\rangle\leq\\
& & \frac{\sigma_{i,0}\|L_i\|^2}{\tau_N^2(\sum_{i=1}^m \sigma_{i,0}\|L_i\|^2)}\|x_{N+1}-x_N\|^2+\frac{\sum_{i=1}^m \sigma_{i,0}\|L_i\|^2}{\sigma_{i,0}}\|v_{i,N}-\ol v_i\|^2
\end{eqnarray*}
and obtain finally
$$ \frac{\lambda\|x_{N+1}-\ol x\|^2}{\tau_{N+1}^2}+\sum_{i=1}^{m}\frac{\|v_{i,N}-\ol{v}_i\|^2}{\tau_1\sigma_{i,0}} \leq  \frac{\lambda\|x_1-\ol x\|^2}{\tau_1^2}+\sum_{i=1}^{m}\frac{\|v_{i,0}-\ol v_i\|^2}{\tau_1\sigma_{i,0}} +\frac{\|x_1-x_0\|^2}{\tau_0^2}$$$$ +\frac{2}{\tau_0}\sum_{i=1}^{m}\langle L_i(x_1-x_0),v_{i,0}-\ol v_i\rangle+\sum_{i=1}^m\frac{\sum_{j=1}^m \sigma_{j,0}\|L_j\|^2}{\sigma_{i,0}}\|v_{i,N}-\ol v_i\|^2,$$
which rapidly yields the inequality in the statement of the theorem.

We close the proof by showing that $\lim\limits_{n\rightarrow+\infty}n\tau_n=\lambda/\gamma$. Notice that for any $n\geq 0$, \begin{equation}\label{tau-n}\tau_{n+1}=\frac{\tau_n}{\sqrt{1+\frac{\tau_n}{\lambda}(2\gamma-\eta\tau_n)}}.\end{equation}
Since $0<\tau_0<2\gamma/\eta$, it follows by induction that $0<\tau_{n+1}<\tau_n<\tau_0<2\gamma/\eta$ for any $n \geq 1$, hence the sequence $(\tau_n)_{n \geq 0}$ converges. In the light of \eqref{tau-n} one easily obtains that $\lim\limits_{n\rightarrow+\infty}\tau_n=0$ and, further, that $\lim\limits_{n\rightarrow+\infty}\frac{\tau_n}{\tau_{n+1}}=1$. As $(\frac{1}{\tau_n})_{n\geq 0}$ is a strictly increasing and unbounded sequence, by applying the Stolz-Ces\`{a}ro Theorem it yields
\begin{align*}
\lim\limits_{n\rightarrow+\infty}n\tau_n & = \lim\limits_{n\rightarrow+\infty}\frac{n}{\frac{1}{\tau_n}}=\lim\limits_{n\rightarrow+\infty}\frac{n+1-n}{\frac{1}{\tau_{n+1}}-\frac{1}{\tau_n}}
=\lim\limits_{n\rightarrow+\infty}\frac{\tau_n\tau_{n+1}}{\tau_n-\tau_{n+1}}\\
& = \lim\limits_{n\rightarrow+\infty}\frac{\tau_n\tau_{n+1}(\tau_n+\tau_{n+1})}{\tau_n^2-\tau_{n+1}^2} =\lim\limits_{n\rightarrow+\infty}\frac{\tau_n\tau_{n+1}(\tau_n+\tau_{n+1})}{\tau_{n+1}^2\frac{\tau_n}{\lambda}(2\gamma-\eta\tau_n)}\\
& = \lim\limits_{n\rightarrow+\infty}\frac{\tau_n+\tau_{n+1}}{\tau_{n+1}(\frac{2\gamma}{\lambda}-\frac{\eta}{\lambda}\tau_n)}=
\lim\limits_{n\rightarrow+\infty}\frac{\frac{\tau_n}{\tau_{n+1}}+1}{\frac{2\gamma}{\lambda}-\frac{\eta}{\lambda}\tau_n}=\frac{\lambda}{\gamma}.
\end{align*}
\end{proof}

\begin{remark}\label{unique-x}
Let us mention that, if $A+C$ is $\gamma$-strongly monotone with $\gamma>0$, then the operator $A+\sum_{i=1}^m L_i^*(B_i(L_i\cdot-r_i))+C$ is strongly monotone, as well, thus the monotone inclusion problem \eqref{sum-k-primal-C} has at most one solution. Hence, if $(\ol x,\ol v_1,...,\ol v_m)$ is a primal-dual solution to Problem \ref{pr2}, then $\ol x$ is the unique solution to \eqref{sum-k-primal-C}.  Notice that the problem \eqref{sum-k-dual-C} may not have an unique solution.
\end{remark}

\subsection{The case $A+C$ and $B_i^{-1}+D_i^{-1}, i=1,..., m,$ are strongly monotone}\label{subsec2.2}

In this subsection we propose a modified version of the algorithm in Theorem \ref{vu} which guarantees when $A+C$ and $B_i^{-1}+D_i^{-1}, i=1,...,m,$ are strongly monotone
orders of convergence of ${\cal{O}}(\omega^n)$, for $\omega \in (0,1)$, for the sequences $(x_n)_{n \geq 0}$ and $(v_{i,n})_{n \geq 0}, i=1,...,m$. The algorithm aims to solve the primal-dual pair of monotone inclusions stated in Problem \ref{pr1} under relaxed assumptions for the operators $C$ and $D_i^{-1}, i=1,...,m$.

\begin{problem}\label{pr3}
Let ${\cal H}$ be a real Hilbert space, $z\in {\cal H}$, $A:{\cal H}\rightrightarrows {\cal H}$ a maximally monotone operator and $C:{\cal H}\rightarrow {\cal H}$ a monotone and $\eta$-Lipschitzian operator for $\eta>0$. Let $m$ be a strictly positive integer and for any $i\in\{1,...,m\}$ let ${\cal G}_i$  be a real Hilbert space, $r_i\in {\cal G}_i$, $B_i:{\cal G }_i \rightrightarrows {\cal G}_i$ a maximally monotone operator, $D_i:{\cal G}_i\rightrightarrows {\cal G}_i$ a monotone operators such that $D_i^{-1}$ is $\nu_i$-Lipschitzian for $\nu_i>0$ and $L_i:{\cal H}\rightarrow$ ${\cal G}_i$ a nonzero linear continuous operator. The problem is to solve the primal inclusion
\begin{equation}\label{sum-k-primal-C-D-acc2}
\mbox{find } \ol x \in {\cal H} \ \mbox{such that} \ z\in A\ol x+ \sum_{i=1}^{m}L_i^*\big((B_i\Box D_i)(L_i \ol x-r_i)\big)+C \ol x,
\end{equation}
together with the dual inclusion
\begin{equation}\label{sum-k-dual-C-D-acc2}
\mbox{ find } \ol v_1 \in {\cal G}_1,...,\ol v_m \in {\cal G}_m \ \mbox{such that } \exists x\in {\cal H}: \
\left\{
\begin{array}{ll}
z-\sum_{i=1}^{m}L_i^*\ol v_i\in Ax+Cx\\
\ol v_i\in (B_i\Box D_i)(L_ix-r_i), i=1,...,m.
\end{array}\right.
\end{equation}
\end{problem}

Under the assumption that $A+C$ is $\gamma$-strongly monotone for $\gamma>0$  and $B_i^{-1}+D_i^{-1}$ is $\delta_i$-strongly monotone with $\delta_i>0, i=1,....m$,  we propose the following modification of the iterative scheme in Theorem \ref{vu}.

\begin{algorithm}\label{alg-k-op-acc2}$ $

\noindent\begin{tabular}{rl}
\verb"Initialization": & \verb"Choose" $\mu >0$ \verb"such that"\\
                       & $\mu\leq \min\left \{\gamma^2/\eta^2,\delta_1^2/\nu_1^2,...,\delta_m^2/\nu_m^2,\sqrt{\gamma/\left(\sum_{i=1}^{m}\|L_i\|^2/\delta_i\right)}\right\}$, \\
                       & $\tau=\mu/(2\gamma)$, $\sigma_i=\mu/(2\delta_i)$, $i=1,..,m$, \\
                       & $\theta\in[2/(2+\mu),1]$ \verb"and" $(x_0,v_{1,0},...,v_{m,0}) \in {\cal H} \times $ ${\cal G}_1$ $\times...\times$ ${\cal G}_m$.\\
\verb"For" $n\geq 0$ \verb"set": & $x_{n+1}=J_{\tau A}\big[x_n-\tau\big(\sum_{i=1}^{m}L_i^*v_{i,n}+Cx_n-z\big)\big]$\\
                                 & $y_n=x_{n+1}+\theta (x_{n+1} - x_n)$\\
                                 & $v_{i,n+1}=J_{\sigma_i B_i^{-1}}[v_{i,n}+\sigma_i(L_iy_n-D_i^{-1}v_{i,n}-r_i)]$, $i=1,...,m.$
\end{tabular}
\end{algorithm}

\begin{theorem}\label{th-k-op-acc2} Suppose that $A+C$ is $\gamma$-strongly monotone for $\gamma>0$, $B_i^{-1}+D_i^{-1}$ is $\delta_i$-strongly monotone for $\delta_i>0$, $i=1,...,m$, and let $(\ol x,\ol v_1,...,\ol v_m)$ be a primal-dual solution to Problem \ref{pr3}. Then the sequences generated by Algorithm \ref{alg-k-op-acc2} fulfill for any $n\geq 0$
\begin{eqnarray*}
& & \!\!\!\! \gamma\|x_{n+1}-\ol x\|^2+(1-\omega)\sum_{i=1}^{m}\delta_i\|v_{i,n}-\ol v_i\|^2\leq\\
& & \!\!\!\! \omega^n \left(\gamma\|x_1- \ol x\|^2+\sum_{i=1}^{m}\delta_i\|v_{i,0}-\ol v_i\|^2+\frac{\gamma}{2}\omega\|x_1-x_0\|^2+\mu\omega\sum_{i=1}^{m}\langle L_i(x_1-x_0),v_{i,0}-\ol v_i\rangle\right),
\end{eqnarray*}
where $0<\omega=\frac{2(1+\theta)}{4+\mu}<1$.
\end{theorem}

\begin{proof}
For any $n \geq 0$ we have
\begin{equation}\label{J-A-acc2}
\frac{1}{\tau}(x_{n+1}-x_{n+2})-\left(\sum_{i=1}^{m}L_i^*v_{i,n+1}+Cx_{n+1}-z\right)+Cx_{n+2}\in (A+C)x_{n+2},
\end{equation}
thus, since $A+C$ is $\gamma$-strongly monotone, \eqref{sum-k-dual-C-D-acc2} yields
\begin{eqnarray*}
& & \gamma\|x_{n+2}-\ol x\|^2\leq  \left\langle x_{n+2}-\ol x, \frac{1}{\tau}(x_{n+1}-x_{n+2}) \right \rangle + \\
& & \left\langle x_{n+2}-\ol x, -\left(\sum_{i=1}^{m}L_i^*v_{i,n+1}+Cx_{n+1}-z\right)+Cx_{n+2}-\left(z-\sum_{i=1}^{m}L_i^*\ol v_i\right)\right\rangle = \\
& & \frac{1}{\tau}\left\langle x_{n+2}-\ol x,x_{n+1}-x_{n+2}\right\rangle+\langle x_{n+2}-\ol x,Cx_{n+2}-Cx_{n+1}\rangle+\\
& & \sum_{i=1}^{m}\left\langle L_i(x_{n+2}-\ol x),\ol v_i-v_{i,n+1}\right\rangle.
\end{eqnarray*}

Further, by using \eqref{hilb} and $$\langle x_{n+2}-\ol x,Cx_{n+2}-Cx_{n+1}\rangle\leq \frac{\gamma}{2}\|x_{n+2}-\ol x\|^2+\frac{\eta^2}{2\gamma}\|x_{n+2}-x_{n+1}\|^2,$$
we get for any $n \geq 0$
\begin{eqnarray*}
& & \left(\frac{1}{2\tau}+\frac{\gamma}{2}\right)\|x_{n+2}-\ol x\|^2 \leq  \\
& & \frac{1}{2\tau}\|x_{n+1}-\ol x\|^2 - \left(\frac{1}{2\tau}-\frac{\eta^2}{2\gamma}\right)\|x_{n+2}-x_{n+1}\|^2 + \sum_{i=1}^{m}\langle L_i(x_{n+2}-\ol x),\ol v_i-v_{i,n+1}\rangle.
\end{eqnarray*}
After multiplying this inequality with $\mu$ and taking into account that
$$\frac{\mu}{2\tau}=\gamma, \mu\left(\frac{1}{2\tau}+\frac{\gamma}{2}\right)=\gamma\left(1+\frac{\mu}{2}\right) \ \mbox{and} \ \mu\left(\frac{1}{2\tau}-\frac{\eta^2}{2\gamma}\right)=\gamma-\frac{\eta^2}{2\gamma}\mu\geq \frac{\gamma}{2},$$
we obtain for any $n \geq 0$
\begin{eqnarray}\label{x-n-acc2}
& & \gamma\left(1+\frac{\mu}{2}\right)\|x_{n+2}-\ol x\|^2 \leq  \\
& & \gamma\|x_{n+1}-\ol x\|^2 - \frac{\gamma}{2}\|x_{n+2}-x_{n+1}\|^2 + \mu \sum_{i=1}^{m}\langle L_i(x_{n+2}-\ol x),\ol v_i-v_{i,n+1}\rangle. \nonumber
\end{eqnarray}

On the other hand, for every $i=1,...,m$ and any $n \geq 0$, from
\begin{equation}\label{J-B-i-acc2} \frac{1}{\sigma_i}(v_{i,n}-v_{i,n+1})+L_iy_n-D_i^{-1}v_{i,n}-r_i+D_i^{-1}v_{i,n+1}\in (B_i^{-1}+D_i^{-1})v_{i,n+1},
\end{equation}
the $\delta_i$-strong monotonicity of $B_i^{-1}+D_i^{-1}$ and \eqref{sum-k-dual-C-D-acc2} we obtain
\begin{eqnarray*}
\delta_i\|v_{i,n+1}- \ol v_i\|^2 & \leq &  \left\langle \frac{1}{\sigma_i}(v_{i,n}-v_{i,n+1}), v_{i,n+1}-\ol v_i\right\rangle \\
                                 & + & \left\langle L_iy_n-r_i-D_i^{-1}v_{i,n}+D_i^{-1}v_{i,n+1}-(L_i\ol x-r_i), v_{i,n+1}-\ol v_i\right\rangle.
\end{eqnarray*}
Further, for every $i=1,...,m$ and any $n \geq 0$ we have
$$\frac{1}{\sigma_i}\langle v_{i,n}-v_{i,n+1},v_{i,n+1}-\ol v_i\rangle=\frac{1}{2\sigma_i}\|v_{i,n}-\ol v_i\|^2-\frac{1}{2\sigma_i}\|v_{i,n}-v_{i,n+1}\|^2-\frac{1}{2\sigma_i}\|v_{i,n+1}-\ol v_i\|^2$$
and, since $D_i^{-1}$ is a $\nu_i$-Lipschitzian operator,
$$\langle D_i^{-1}v_{i,n+1}-D_i^{-1}v_{i,n}, v_{i,n+1}-\ol v_i\rangle\leq \frac{\delta_i}{2}\|v_{i,n+1}-\ol v_i\|^2+\frac{\nu_i^2}{2\delta_i}\|v_{i,n+1}-v_{i,n}\|^2.$$
Consequently, for every $i=1,...,m$ and any $n \geq 0$ it holds
\begin{eqnarray*}
& & \left(\frac{1}{2\sigma_i}+\frac{\delta_i}{2}\right)\|v_{i,n+1}-\ol v_i\|^2 \leq  \\
& & \frac{1}{2\sigma_i}\|v_{i,n}-\ol v_i\|^2 - \left(\frac{1}{2\sigma_i}-\frac{\nu_i^2}{2\delta_i}\right)\|v_{i,n+1}-v_{i,n}\|^2 + \langle L_i(\ol x-y_n),\ol v_i-v_{i,n+1}\rangle,
\end{eqnarray*}
which, after multiplying it by $\mu$ (here is the initial choice of $\mu$ determinant), yields
\begin{equation}\label{v-n-acc2}
\delta_i\left(1+\frac{\mu}{2}\right)\|v_{i,n+1}-\ol v_i\|^2 \leq  \delta_i\|v_{i,n}-\ol v_i\|^2 - \frac{\delta_i}{2}\|v_{i,n+1}-v_{i,n}\|^2 + \mu \langle L_i(\ol x-y_n),\ol v_i-v_{i,n+1}\rangle.
\end{equation}
We denote
$$a_n:=\gamma\|x_{n+1}-\ol x\|^2+\sum_{i=1}^{m}\delta_i\|v_{i,n}-\ol v_i\|^2 \ \forall n \geq 0.$$
Summing up the inequalities in \eqref{x-n-acc2} and \eqref{v-n-acc2} we obtain for any $n \geq 0$
\begin{eqnarray}\label{an,n+1}
& & \left(1+\frac{\mu}{2}\right)a_{n+1} \leq  a_n\\
& & -\frac{\gamma}{2}\|x_{n+2}- x_{n+1}\|^2-\sum_{i=1}^{m}\frac{\delta_i}{2}\|v_{i,n}- v_{i,n+1}\|^2 + \mu \sum_{i=1}^{m}\langle L_i(x_{n+2}-y_n),\ol v_i-v_{i,n+1}\rangle.\nonumber
\end{eqnarray}
Further, since $y_n=x_{n+1}+\theta(x_{n+1}- x_n)$ and $\omega \leq \theta$, for every $i=1,...,m$ and any $n \geq 0$ it holds
\begin{eqnarray*}
& & \langle L_i(x_{n+2}-y_n),\ol v_i - v_{i,n+1}\rangle=\langle L_i\left(x_{n+2}-x_{n+1} - \theta(x_{n+1}- x_n)\right),\ol v_i - v_{i,n+1}\rangle=\\
& & \langle L_i(x_{n+2}-x_{n+1}),\ol v_i - v_{i,n+1}\rangle-\omega\langle L_i(x_{n+1}-x_n),\ol v_i - v_{i,n}\rangle+ \\
& & \omega\langle L_i(x_{n+1}-x_n),v_{i,n+1}- v_{i,n}\rangle + (\theta-\omega)\langle L_i(x_{n+1}-x_n),v_{i,n+1}-\ol v_i\rangle \leq\\
& & \langle L_i(x_{n+2}-x_{n+1}),\ol v_i - v_{i,n+1}\rangle-\omega\langle L_i(x_{n+1}-x_n),\ol v_i - v_{i,n}\rangle+\\
& & \omega \|L_i\|\left(\mu\omega\|L_i\|\frac{\|x_{n+1}-x_n\|^2}{2\delta_i}+\delta_i\frac{\|v_{i,n+1}-v_{i,n}\|^2}{2\mu\omega\|L_i\|}\right) + \\
& & (\theta-\omega)\|L_i\|\left(\mu\omega\|L_i\|\frac{\|x_{n+1}-x_n\|^2}{2\delta_i}+\delta_i\frac{\|v_{i,n+1}-\ol v_i\|^2}{2\mu\omega\|L_i\|}\right)=\\
& & \langle L_i(x_{n+2}-x_{n+1}),\ol v_i - v_{i,n+1}\rangle-\omega\langle L_i(x_{n+1}-x_n),\ol v_i - v_{i,n}\rangle+\\
& & \theta \mu \omega \|L_i\|^2\frac{\|x_{n+1}-x_n\|^2}{2\delta_i} + \delta_i\frac{\|v_{i,n+1}-v_{i,n}\|^2}{2\mu} + (\theta-\omega) \delta_i \frac{\|v_{i,n+1}-\ol v_i\|^2}{2\mu \omega}.
\end{eqnarray*}
Taking into consideration that
$$\frac{\mu^2\theta\omega}{2}\sum_{i=1}^m\frac{\|L_i\|^2}{\delta_i}\leq \frac{\gamma\theta}{2}\omega\leq \frac{\gamma}{2}\omega \ \mbox{and} \ 1+\frac{\mu}{2}=\frac{1}{\omega}+\frac{\theta-\omega}{\omega},$$
from \eqref{an,n+1} we obtain for any $n \geq 0$
\begin{eqnarray*}
& & \frac{1}{\omega}a_{n+1}  + \frac{\gamma}{2}\|x_{n+2}- x_{n+1}\|^2 \leq\\
& & a_n + \frac{\gamma}{2}\omega \|x_{n+1}-x_{n}\|^2 - \frac{\theta-\omega}{\omega}\left(a_{n+1}-\sum_{i=1}^{m}\frac{\delta_i}{2}\|v_{i,n+1}-\ol v_i\|^2\right) + \\
& & \mu\sum_{i=1}^{m}\big(\langle L_i(x_{n+2}-x_{n+1}),\ol v_i - v_{i,n+1}\rangle-\omega\langle L_i(x_{n+1}-x_n),\ol v_i - v_{i,n}\rangle\big).
\end{eqnarray*}
As $\omega \leq \theta$ and $a_{n+1}-\sum_{i=1}^{m}\frac{\delta_i}{2}\|v_{i,n+1}-\ol v_i\|^2\geq 0$, we further get after multiplying the last inequality with $\omega^{-n}$ the following Fej\'er-type inequality that holds for any $n \geq 0$
\begin{eqnarray}\label{fej2}
& & \omega^{-(n+1)}a_{n+1}+\frac{\gamma}{2}\omega^{-n}\|x_{n+2}-x_{n+1}\|^2 + \mu \omega^{-n} \sum_{i=1}^{m} \langle L_i(x_{n+2}-x_{n+1}),v_{i,n+1}-\ol v_i\rangle \leq \nonumber\\
& & \omega^{-n}a_n + \frac{\gamma}{2}\omega^{-(n-1)}\|x_{n+1}-x_{n}\|^2 + \mu\omega^{-(n-1)}\sum_{i=1}^{m} \langle L_i(x_{n+1}-x_n),v_{i,n}-\ol v_i\rangle.
\end{eqnarray}
Let $N\in\N, N\geq 2$. Summing up the inequality in \eqref{fej2} from $n=0$ to $N-1$, it yields
\begin{eqnarray*}
& & \omega^{-N}a_N+\frac{\gamma}{2}\omega^{-N+1}\|x_N-x_{N+1}\|^2 + \mu \omega^{-N+1} \sum_{i=1}^{m} \langle L_i(x_{N+1}-x_{N}),v_{i,N}-\ol v_i\rangle \leq\\
& & a_0 + \frac{\gamma}{2}\omega\|x_1-x_0\|^2 + \mu \omega \sum_{i=1}^{m} \langle L_i(x_{1}-x_{0}),v_{i,0}-\ol v_i\rangle.
\end{eqnarray*}
Using that
$$\langle L_i(x_{N+1}-x_N),v_{i,N}-\ol v_i\rangle\geq -\frac{\mu\|L_i\|^2}{4\delta_i}\|x_{N+1}-x_N\|^2-\frac{\delta_i}{\mu}\|v_{i,N}-\ol v_i\|^2, \ i=1,...,m,$$
this further yields
\begin{eqnarray}\label{eqfin}
& & \omega^{-N}a_N+\omega^{-N+1}\left(\frac{\gamma}{2}-\frac{\mu^2}{4}\sum_{i=1}^{m}\frac{\|L_i\|^2}{\delta_i}\right)\|x_N-x_{N+1}\|^2 -\omega^{-N+1}\sum_{i=1}^{m}\delta_i\|v_{i,N}-\ol v_i\|^2 \leq \nonumber\\
& & a_0 + \frac{\gamma}{2}\omega\|x_1-x_0\|^2 + \mu \omega \sum_{i=1}^{m} \langle L_i(x_{1}-x_{0}),v_{i,0}-\ol v_i\rangle.
\end{eqnarray}
Taking into account the way $\mu$ has been chosen, we have
$$\frac{\gamma}{2}-\frac{\mu^2}{4}\sum_{i=1}^{m}\frac{\|L_i\|^2}{\delta_i}\geq \frac{\gamma}{2}-\frac{\gamma}{4}>0,$$ hence, after multiplying \eqref{eqfin}  with $\omega^{-N}$ it yields
\begin{equation*}
 a_N -\omega \sum_{i=1}^{m}\delta_i\|v_{i,N}-\ol v_i\|^2 \leq  \omega^N \left(a_0 + \frac{\gamma}{2}\omega\|x_1-x_0\|^2 + \mu \omega \sum_{i=1}^{m} \langle L_i(x_{1}-x_{0}),v_{i,0}-\ol v_i\rangle \right).
\end{equation*}
The conclusion follows by taking into account the definition of the sequence $(a_n)_{n \geq 0}$.
\end{proof}

\begin{remark}\label{unique-x-v} If $A+C$ is $\gamma$-strongly monotone for $\gamma>0$ and $B_i^{-1}+D_i^{-1}$ is $\delta_i$-strongly monotone for $\delta_i>0$, $i=1,...,m$, then there exists at most one primal-dual solution to Problem \ref{pr3}. Hence, if $(\ol x,\ol v_1,...,\ol v_m)$ is a primal-dual solution to Problem \ref{pr3}, then $\ol x$ is the unique solution to the primal inclusion \eqref{sum-k-primal-C-D-acc2} and $(\ol v_1,...,\ol v_m)$ is the unique solution to the dual inclusion \eqref{sum-k-dual-C-D-acc2}.
\end{remark}

\section{Convex optimization problems}\label{sec-opt-pb}

The aim of this section is to show that the two algorithms proposed in this paper and investigated from the point of view of their convergence properties can be employed when solving a primal-dual pair of convex optimization problems.

For a function $f:{\cal H}\rightarrow\overline{\R}$, where $\overline{\R}:=\R\cup\{\pm\infty\}$ is the extended real line, we denote by $\dom f=\{x\in {\cal H}:f(x)<+\infty\}$ its \textit{effective domain} and say that $f$ is \textit{proper} if $\dom f\neq\emptyset$ and $f(x)\neq-\infty$ for all $x\in {\cal H}$. We denote by $\Gamma({\cal H})$ the family of proper convex and lower semi-continuous extended real-valued functions defined on ${\cal H}$. Let $f^*:{\cal H} \rightarrow \overline \R$, $f^*(u)=\sup_{x\in {\cal H}}\{\langle u,x\rangle-f(x)\}$ for all $u\in {\cal H}$, be the \textit{conjugate function} of $f$. The \textit{subdifferential} of $f$ at $x\in {\cal H}$, with $f(x)\in\R$, is the set $\partial f(x):=\{v\in {\cal H}:f(y)\geq f(x)+\langle v,y-x\rangle \ \forall y\in {\cal H}\}$. We take by convention $\partial f(x):=\emptyset$, if $f(x)\in\{\pm\infty\}$.  Notice that if $f\in\Gamma({\cal H})$, then $\partial f$ is a maximally monotone operator (cf. \cite{rock}) and it holds $(\partial f)^{-1} = \partial f^*$. For $f,g:{\cal H}\rightarrow \overline{\R}$ two proper functions, we consider their \textit{infimal convolution}, which is the function $f\Box g:{\cal H}\rightarrow\B$, defined by $(f\Box g)(x)=\inf_{y\in {\cal H}}\{f(y)+g(x-y)\}$, for all $x\in {\cal H}$.

Let $S\subseteq {\cal H}$ be a nonempty set. The \textit{indicator function} of $S$, $\delta_S:{\cal H}\rightarrow \overline{\R}$, is the function which takes the value $0$ on $S$ and $+\infty$ otherwise. The subdifferential of the indicator function is the \textit{normal cone} of $S$, that is $N_S(x)=\{u\in {\cal H}:\langle u,y-x\rangle\leq 0 \ \forall y\in S\}$, if $x\in S$ and $N_S(x)=\emptyset$ for $x\notin S$.

When $f\in\Gamma({\cal H})$ and $\gamma > 0$, for every $x \in {\cal H}$ we denote by $\prox_{\gamma f}(x)$ the \textit{proximal point} of parameter $\gamma$ of $f$ at $x$, which is the unique optimal solution of the optimization problem
\begin{equation}\label{prox-def}\inf_{y\in {\cal H}}\left \{f(y)+\frac{1}{2\gamma}\|y-x\|^2\right\}.
\end{equation}
Notice that $J_{\gamma\partial f}=(\id_{\cal H}+\gamma\partial f)^{-1}=\prox_{\gamma f}$, thus  $\prox_{\gamma f} :{\cal H} \rightarrow {\cal H}$ is a single-valued operator fulfilling the extended \textit{Moreau's decomposition formula}
\begin{equation}\label{prox-f-star}
\prox\nolimits_{\gamma f}+\gamma\prox\nolimits_{(1/\gamma)f^*}\circ\gamma^{-1}\id\nolimits_{\cal H}=\id\nolimits_{\cal H}.
\end{equation}
Let us also recall that the function $f:{\cal H} \rightarrow \overline \R$ is said to be \textit{$\gamma$-strongly convex} for $\gamma >0$, if $f-\frac{\gamma}{2}\|\cdot\|^2$ is a convex function. Let us mention that this property implies $\gamma$-strong monotonicity of $\partial f$ (see \cite[Example 22.3]{bauschke-book}).

Finally, we notice that for $f=\delta_S$, where $S\subseteq {\cal H}$ is a nonempty convex and closed set, it holds
\begin{equation}\label{projection}
J_{\gamma N_S}=J_{N_S}=J_{\partial \delta_S} = (\id\nolimits_{\cal H}+N_S)^{-1}=\prox\nolimits_{\delta_S}=P_S,
\end{equation}
where  $P_S :{\cal H} \rightarrow C$ denotes the \textit{projection operator} on $S$ (see \cite[Example 23.3 and Example 23.4]{bauschke-book}).

In order to investigate the applicability of the algorithm introduced in Subsection \ref{subsec2.1} we consider the following primal-dual pair of convex optimization problems.

\begin{problem}\label{pr4} Let ${\cal H}$ be a real Hilbert space, $z\in {\cal H}$, $f\in\Gamma({\cal H})$ and $h:{\cal H}\rightarrow \R$ a convex and differentiable function with a $\eta$-Lipschitzian gradient for $\eta>0$. Let $m$ be a strictly positive integer and for any $i\in\{1,...,m\}$ let ${\cal G}_i$  be a real Hilbert space, $r_i\in {\cal G}_i$, $g_i\in\Gamma({\cal G}_i)$ and $L_i:{\cal H}\rightarrow$ ${\cal G}_i$ a nonzero linear continuous operator. Consider the convex optimization problem
\begin{equation}\label{sum-k-prim-f}
\inf_{x\in {\cal H}}\left\{f(x)+\sum_{i=1}^{m}g_i(L_ix-r_i)+h(x)-\langle x,z\rangle\right\}
\end{equation}
and its \textit{Fenchel-type dual} problem
\begin{equation}\label{sum-k-dual-f}
\sup_{v_i\in {\cal{G}}_i, i=1,...,m}\left\{-\big(f^*\Box h^*\big)\left(z-\sum_{i=1}^{m}L_i^*v_i\right)-\sum_{i=1}^{m}\big(g_i^*(v_i)+\langle v_i,r_i\rangle\big) \right\}.
\end{equation}
\end{problem}

Considering maximal monotone operators
$$A=\partial f, C=\nabla h \ \mbox{and} \ B_i=\partial g_i, i=1,...,m,$$
the monotone inclusion problem \eqref{sum-k-primal-C} reads
\begin{equation}\label{sum-k-primal-C-f}
\mbox{find } \ol x \in {\cal H} \ \mbox{such that} \ z\in \partial f(\ol x)+ \sum_{i=1}^{m}L_i^*(\partial g_i(L_i \ol x-r_i))+ \nabla h(\ol x),
\end{equation}
while the dual inclusion problem \eqref{sum-k-dual-C} reads
\begin{equation}\label{sum-k-dual-C-f}
\mbox{ find } \ol v_1 \in {\cal G}_1,...,\ol v_m \in {\cal G}_m \ \mbox{such that } \exists x\in {\cal H}: \
\left\{
\begin{array}{ll}
z-\sum_{i=1}^{m}L_i^*\ol v_i\in \partial f(x)+\nabla h(x)\\
\ol v_i\in \partial g_i(L_ix-r_i), i=1,...,m.
\end{array}\right.
\end{equation}

If $(\ol x, \ol v_1,...,\ol v_m)\in{\cal H} \times$ ${\cal{G}}_1 \times...\times {\cal{G}}_m$ is a primal-dual solution to \eqref{sum-k-primal-C-f}-\eqref{sum-k-dual-C-f}, namely,
\begin{equation}\label{prim-dual-f}z-\sum_{i=1}^{m}L_i^*\ol v_i\in \partial f(\ol x)+\nabla h(\ol x) \mbox{ and }\ol v_i\in \partial g_i(L_i\ol x-r_i), i=1,...,m,\end{equation}
then $\ol x$ is an optimal solution of the problem \eqref{sum-k-prim-f}, $(\ol v_1,...,\ol v_m)$ is an optimal solution of \eqref{sum-k-dual-f} and the optimal objective values of the two problems coincide. Notice that \eqref{prim-dual-f} is nothing else than the system of optimality conditions for the primal-dual pair of convex optimization problems \eqref{sum-k-prim-f}-\eqref{sum-k-dual-f}.

In case a qualification condition is fulfilled, these optimality conditions are also necessary. For the readers convenience, let us present some qualification conditions which are suitable in this context. One of the weakest qualification conditions of interiority-type reads (see, for instance, \cite[Proposition 4.3, Remark 4.4]{combettes-pesquet})
\begin{equation}\label{reg-cond} (r_1,...,r_m)\in\sqri\left(\prod_{i=1}^{m}\dom g_i-\{(L_1x,...,L_mx):x\in \dom f\}\right).
\end{equation}
Here, for ${\cal H}$ a real Hilbert space and $S\subseteq {\cal H}$ a convex set,  we denote by
$$\sqri S:=\{x\in S:\cup_{\lambda>0}\lambda(S-x) \ \mbox{is a closed linear subspace of} \ {\cal H}\}$$
its \textit{strong quasi-relative interior}. Notice that we always have $\inte S\subseteq\sqri S$ (in general this inclusion may be strict). If ${\cal H}$ is finite-dimensional, then $\sqri S$ coincides with $\ri S$, the relative interior of $S$, which is the interior of $S$ with respect to its affine hull.
The condition \eqref{reg-cond} is fulfilled if (i) $\dom g_i={\cal{G}}_i$, $i=1,...,m$ or (ii) ${\cal H}$ and ${\cal{G}}_i$ are finite-dimensional and there exists $x\in\ri\dom f$ such that $L_ix-r_i\in\ri\dom g_i$, $i=1,...,m$ (see \cite[Proposition 4.3]{combettes-pesquet}). Another useful and easily verifiable qualification condition guaranteing the optimality conditions \eqref{prim-dual-f} has the following formulation: there exists $x'\in \dom f \cap \bigcap_{i=1}^m L_i^{-1}(r_i+\dom g_i)$ such that $g_i$ is continuous at $L_ix'-r_i$, $i=1,...,m$ (see \cite[Remark 2.5]{b-hab} and \cite{b-c-h}).

The following two statements are particular instances of Algorithm \ref{alg-k-op-acc1} and Theorem \ref{th-k-op-acc1}, respectively.

\begin{algorithm}\label{alg-f}$ $

\noindent\begin{tabular}{rl}
\verb"Initialization": & \verb"Choose" $\tau_0>0, \sigma_{i,0}>0$, $i=1,..,m$, \verb"such that"\\
                       & $\tau_0<2\gamma/\eta$, $\lambda\geq \eta+1$, $\tau_0\sum_{i=1}^m \sigma_{i,0}\|L_i\|^2\leq \sqrt{1+\tau_0(2\gamma-\eta\tau_0)/\lambda}$\\
                       & \verb"and" $(x_0,v_{1,0},...,v_{m,0}) \in {\cal H} \times $ ${\cal G}_1$ $\times...\times$ ${\cal G}_m$. \\
\verb"For" $n\geq 0$ \verb"set": & $x_{n+1}=\prox_{(\tau_n/\lambda)  f}\big[x_n-(\tau_n/\lambda)\big(\sum_{i=1}^{m}L_i^*v_{i,n}+\nabla h (x_n)-z\big)\big]$\\
                                 & $\theta_n=1/\sqrt{1+\tau_n(2\gamma-\eta\tau_n)/\lambda}$\\
                                 & $y_n=x_{n+1}+\theta_n (x_{n+1}- x_n)$\\
                                 & $v_{i,n+1}=\prox_{\sigma_{i,n} g_i^*}[v_{i,n}+\sigma_{i,n}(L_iy_n-r_i)]$, $i=1,...,m$\\
                                 & $\tau_{n+1}=\theta_n\tau_n$, $\sigma_{i,n+1}=\sigma_{i,n}/\theta_{n+1}$, $i=1,...,m.$
\end{tabular}
\end{algorithm}

\begin{theorem}\label{th-f} Suppose that $f+h$ is $\gamma$-strongly convex for $\gamma>0$ and the qualification condition \eqref{reg-cond} holds. Then there exists a unique optimal solution $\ol x$ to \eqref{sum-k-prim-f}, an optimal solution $(\ol v_1,...,\ol v_m)$ to \eqref{sum-k-dual-f} fulfilling the optimality conditions \eqref{prim-dual-f} and such that the optimal objective values of the problems \eqref{sum-k-prim-f} and \eqref{sum-k-dual-f} coincide. The sequences generated by Algorithm \ref{alg-f} fulfill for any $n\geq 0$
\begin{eqnarray*}
& & \frac{\lambda\|x_{n+1}-\ol x\|^2}{\tau_{n+1}^2}+\left(1-\tau_1\sum_{i=1}^m \sigma_{i,0}\|L_i\|^2\right)\sum_{i=1}^{m}\frac{\|v_{i,n}-\ol{v}_i\|^2}{\tau_1\sigma_{i,0}} \leq\\
& & \frac{\lambda\|x_1-\ol x\|^2}{\tau_1^2}+\sum_{i=1}^{m}\frac{\|v_{i,0}-\ol v_i\|^2}{\tau_1\sigma_{i,0}} +\frac{\|x_1-x_0\|^2}{\tau_0^2} + \frac{2}{\tau_0}\sum_{i=1}^{m}\langle L_i(x_1-x_0),v_{i,0}-\ol v_i\rangle.\end{eqnarray*}
Moreover, $\lim\limits_{n\rightarrow+\infty}n\tau_n=\frac{\lambda}{\gamma}$, hence one obtains for $(x_n)_{n \geq 0}$ an order of convergence of ${\cal {O}}(\frac{1}{n})$.
\end{theorem}

\begin{remark}\label{rmun}
The uniqueness of the solution of \eqref{sum-k-prim-f} in the above theorem follows from \cite[Corollary 11.16]{bauschke-book}.
\end{remark}

\begin{remark}\label{h=0} In case $h(x)=0$ for all $x\in {\cal H}$, one has to choose in Algorithm \ref{alg-f} as initial points $\tau_0>0, \sigma_{i,0}>0$, $i=1,..,m$, with $\tau_0\sum_{i=1}^m \sigma_{i,0}\|L_i\|^2\leq \sqrt{1+2\tau_0\gamma/\lambda}$ and $\lambda\geq 1$ and to update the sequence $(\theta_n)_{n \geq 0}$ via $\theta_n=1/\sqrt{1+2\tau_n\gamma/\lambda}$ for any $n \geq 0$, in order to obtain a suitable iterative scheme for solving the pair of primal-dual optimization problems \eqref{sum-k-prim-f}-\eqref{sum-k-dual-f} with the same convergence behavior as of Algorithm \ref{alg-f}.
\end{remark}

We turn now our attention to the algorithm introduced in Subsection \ref{subsec2.2} and consider to this end the following primal-dual pair of convex optimization problems.
\begin{problem}\label{pr5} Let ${\cal H}$ be a real Hilbert space, $z\in {\cal H}$, $f\in\Gamma({\cal H})$ and $h:{\cal H}\rightarrow \R$ a convex and differentiable function with a $\eta$-Lipschitzian gradient for $\eta>0$. Let $m$ be a strictly positive integer and for any $i\in\{1,...,m\}$ let ${\cal G}_i$  be a real Hilbert space, $r_i\in {\cal G}_i$, $g_i, l_i \in\Gamma({\cal G}_i)$ such that $l_i$ is $\nu_i^{-1}$-strongly convex for $\nu_i > 0$ and $L_i:{\cal H}\rightarrow$ ${\cal G}_i$ a nonzero linear continuous operator. Consider the convex optimization problem
\begin{equation}\label{sum-k-prim-f2}
\inf_{x\in {\cal H}}\left\{f(x)+\sum_{i=1}^{m}(g_i \Box l_i)(L_ix-r_i)+h(x)-\langle x,z\rangle\right\}
\end{equation}
and its \textit{Fenchel-type dual} problem
\begin{equation}\label{sum-k-dual-f2}
\sup_{v_i\in {\cal{G}}_i, i=1,...,m}\left\{-\big(f^*\Box h^*\big)\left(z-\sum_{i=1}^{m}L_i^*v_i\right)-\sum_{i=1}^{m}\big(g_i^*(v_i)+l_i^*(v_i) + \langle v_i,r_i\rangle\big) \right\}.
\end{equation}
\end{problem}

Considering the maximal monotone operators
$$A=\partial f, C=\nabla h, B_i=\partial g_i \ \mbox{and} \ D_i=\partial l_i, i=1,...,m,$$
according to \cite[Proposition 17.10, Theorem 18.15]{bauschke-book}, $D_i^{-1} = \nabla l_i^*$ is a monotone and $\nu_i$-Lipschitzian operator for $i=1,...,m$. The monotone inclusion problem \eqref{sum-k-primal-C-D-acc2} reads
\begin{equation}\label{sum-k-primal-C-D-f}
\mbox{find } \ol x \in {\cal H} \ \mbox{such that} \ z\in \partial f(\ol x)+ \sum_{i=1}^{m}L_i^*((\partial g_i \Box \partial l_i) (L_i \ol x-r_i))+ \nabla h(\ol x),
\end{equation}
while the dual inclusion problem \eqref{sum-k-dual-C-D-acc2} reads
\begin{equation}\label{sum-k-dual-C-D-f}
\mbox{ find } \ol v_1 \in {\cal G}_1,...,\ol v_m \in {\cal G}_m \ \mbox{such that } \exists x\in {\cal H}: \
\left\{
\begin{array}{ll}
z-\sum_{i=1}^{m}L_i^*\ol v_i\in \partial f(x)+\nabla h(x)\\
\ol v_i\in (\partial g_i \Box \partial l_i) (L_ix-r_i), i=1,...,m.
\end{array}\right.
\end{equation}

If $(\ol x, \ol v_1,...,\ol v_m)\in{\cal H} \times$ ${\cal{G}}_1 \times...\times {\cal{G}}_m$ is a primal-dual solution to \eqref{sum-k-primal-C-D-f}-\eqref{sum-k-dual-C-D-f}, namely,
\begin{equation}\label{prim-dual-f2}z-\sum_{i=1}^{m}L_i^*\ol v_i\in \partial f(\ol x)+\nabla h(\ol x) \mbox{ and }\ol v_i\in (\partial g_i \Box \partial l_i)(L_i\ol x-r_i), i=1,...,m,\end{equation}
then $\ol x$ is an optimal solution of the problem \eqref{sum-k-prim-f2}, $(\ol v_1,...,\ol v_m)$ is an optimal solution of \eqref{sum-k-dual-f2} and the optimal objective values of the two problems coincide. Notice that \eqref{prim-dual-f2} is nothing else than the system of optimality condition for the primal-dual pair of convex optimization problems \eqref{sum-k-prim-f2}-\eqref{sum-k-dual-f2}.

The assumptions made on $l_i$ guarantees that $g_i \Box l_i \in \Gamma({\cal G}_i)$ (see \cite[Corollary 11.16, Proposition 12.14]{bauschke-book}) and, since $\dom (g_i \Box l_i) = \dom g_i + \dom l_i, i=1,...,m$, one can can consider the following qualification condition of interiority-type in order to guarantee  \eqref{prim-dual-f2}
\begin{equation}\label{reg-cond2}
(r_1,...,r_m)\in\sqri\left(\prod_{i=1}^{m}(\dom g_i + \dom l_i)-\{(L_1x,...,L_mx):x\in \dom f\}\right).
\end{equation}

The following two statements are particular instances of Algorithm \ref{alg-k-op-acc2} and Theorem \ref{th-k-op-acc2}, respectively.

\begin{algorithm}\label{alg-f2}$ $

\noindent\begin{tabular}{rl}
\verb"Initialization": & \verb"Choose" $\mu >0$ \verb"such that"\\
                       & $\mu\leq \min\left \{\gamma^2/\eta^2,\delta_1^2/\nu_1^2,...,\delta_m^2/\nu_m^2,\sqrt{\gamma/\left(\sum_{i=1}^{m}\|L_i\|^2/\delta_i\right)}\right\}$, \\
                       & $\tau=\mu/(2\gamma)$, $\sigma_i=\mu/(2\delta_i)$, $i=1,..,m$, \\
                       & $\theta\in[2/(2+\mu),1]$ \verb"and" $(x_0,v_{1,0},...,v_{m,0}) \in {\cal H} \times $ ${\cal G}_1$ $\times...\times$ ${\cal G}_m$.\\
\verb"For" $n\geq 0$ \verb"set": & $x_{n+1}=\prox_{\tau f}\big[x_n-\tau\big(\sum_{i=1}^{m}L_i^*v_{i,n}+\nabla h(x_n)-z\big)\big]$\\
                                 & $y_n=x_{n+1}+\theta (x_{n+1} - x_n)$\\
                                 & $v_{i,n+1}=\prox_{\sigma_i g_i^*}[v_{i,n}+\sigma_i(L_iy_n-\nabla l_i^*(v_{i,n})-r_i)]$, $i=1,...,m.$
\end{tabular}
\end{algorithm}

\begin{theorem}\label{th-f2} Suppose that $f+h$ is $\gamma$-strongly convex for $\gamma>0$, $g_i^* + l_i^*$ is $\delta_i$-strongly convex for $\delta_i > 0$, $i=1,...,m$, and the qualification condition \eqref{reg-cond2} holds. Then there exists a unique optimal solution $\ol x$ to \eqref{sum-k-prim-f2}, a unique optimal solution $(\ol v_1,...,\ol v_m)$ to \eqref{sum-k-dual-f2} fulfilling the optimality conditions \eqref{prim-dual-f2} and such that the optimal objective values of the problems \eqref{sum-k-prim-f2} and \eqref{sum-k-dual-f2} coincide. The sequences generated by Algorithm \ref{alg-f2} fulfill for any $n\geq 0$
\begin{eqnarray*}
& & \!\!\!\! \gamma\|x_{n+1}-\ol x\|^2+(1-\omega)\sum_{i=1}^{m}\delta_i\|v_{i,n}-\ol v_i\|^2\leq\\
& & \!\!\!\!\omega^n \left(\gamma\|x_1- \ol x\|^2+\sum_{i=1}^{m}\delta_i\|v_{i,0}-\ol v_i\|^2+\frac{\gamma}{2}\omega\|x_1-x_0\|^2+\mu\omega\sum_{i=1}^{m}\langle L_i(x_1-x_0),v_{i,0}-\ol v_i\rangle\right),
\end{eqnarray*}
where $0<\omega=\frac{2(1+\theta)}{4+\mu}<1$.
\end{theorem}

\section{Numerical experiments}

In this section we illustrate the applicability of the theoretical results in the context of two numerical experiments in image processing and support vector machines classification.

\subsection{Image processing}\label{subsec4.1}

In this subsection we compare the numerical performances of Algorithm \ref{alg-f} with the ones in the iterative scheme in Theorem \ref{vu} for an image denoising problem.  To this end we treat the nonsmooth regularized convex optimization problem
\begin{align}\label{probimageproc}
\inf_{x \in \left[0,1\right]^k} \left\{ \frac{1}{2}\left\| x-b \right\|^2 + \lambda_1 TV(x) + \lambda_2\|Wx\|_1 \right\},
\end{align}
where $TV:\R^k \rightarrow \R$ denotes a discrete isotropic total variation, $W:\R^k\rightarrow\R^k$ a the discrete Haar wavelet transform with four levels, $\lambda_1, \lambda_2 >0$ the regularization parameters and $b\in\R^k$ the observed noisy image. Notice that we consider images of size $k=M \times N$ as vectors $x\in\R^k$, where each pixel denoted by $x_{i,j}$, $1\leq i \leq M$, $1\leq j \leq N$, ranges in the closed interval from $0$ (pure black) to $1$ (pure white).

Two popular choices for the discrete total variation functional are the isotropic total variation $TV_{\iso}:\R^k \rightarrow \R$,
\begin{align*}
	TV_{\iso}(x) &= \sum_{i=1}^{M-1}\sum_{j=1}^{N-1}\sqrt{ (x_{i+1,j}-x_{i,j})^2 + (x_{i,j+1}-x_{i,j})^2 } \\
								&\quad + \sum_{i=1}^{M-1} \left| x_{i+1,N}-x_{i,N} \right|  + \sum_{j=1}^{N-1} \left| x_{M,j+1}-x_{M,j} \right| ,
\end{align*}
and the anisotropic total variation $TV_{\aniso}:\R^k \rightarrow \R$,
\begin{align*}
	TV_{\aniso}(x) &= \sum_{i=1}^{M-1}\sum_{j=1}^{N-1} \left|x_{i+1,j}-x_{i,j}\right| + \left|x_{i,j+1}-x_{i,j}\right| \\
								&\quad  + \sum_{i=1}^{M-1} \left| x_{i+1,N}-x_{i,N} \right| + \sum_{j=1}^{N-1} \left| x_{M,j+1}-x_{M,j} \right| ,
\end{align*}
where in both cases reflexive (Neumann) boundary conditions are assumed. Obviously, in both situations the qualification condition in Theorem \ref{th-f} is fulfilled.

Denote $\Y=\R^k \times \R^k$ and define the linear operator $L:\R^k \rightarrow \Y$, $x_{i,j} \mapsto (L_1x_{i,j}, L_2x_{i,j})$, where
\begin{align*}
	L_1x_{i,j} = \left\{ \begin{array}{ll} x_{i+1,j}-x_{i,j}, & \text{if }i<M\\ 0, &\text{if }i=M\end{array}\right. \ \mbox{and} \
	L_2x_{i,j} = \left\{ \begin{array}{ll} x_{i,j+1}-x_{i,j}, & \text{if }j<N\\ 0, &\text{if }j=N\end{array}\right. .
\end{align*}
The operator $L$ represents a discretization of the gradient in horizontal and vertical direction. One can easily check that $\| L \|^2 \leq 8$ and that its adjoint $L^* : \Y \rightarrow \R^k$ is as easy to implement as the operator itself (cf. \cite{Cha04}). Moreover, since $W$ is an orthogonal wavelet, it holds $\|W\|=1$.

\begin{figure}[ht]\label{fig1}
\centering
\subfloat[Noisy image, $\sigma=0.06$]{
\includegraphics[width = 0.31\textwidth]{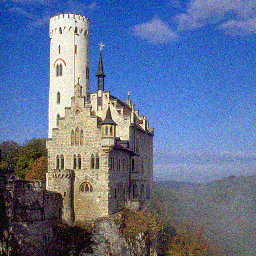}}
\hspace*{0.5cm}
\subfloat[Noisy image, $\sigma=0.12$]{
\includegraphics[width = 0.31\textwidth]{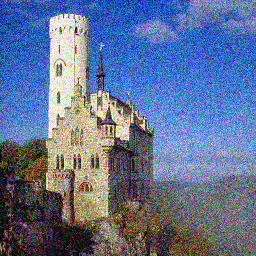}}\\
\subfloat[Denoised image, $\lambda_1=0.035$, $\lambda_2=0.01$, isotropic TV]{
\includegraphics[width = 0.31\textwidth]{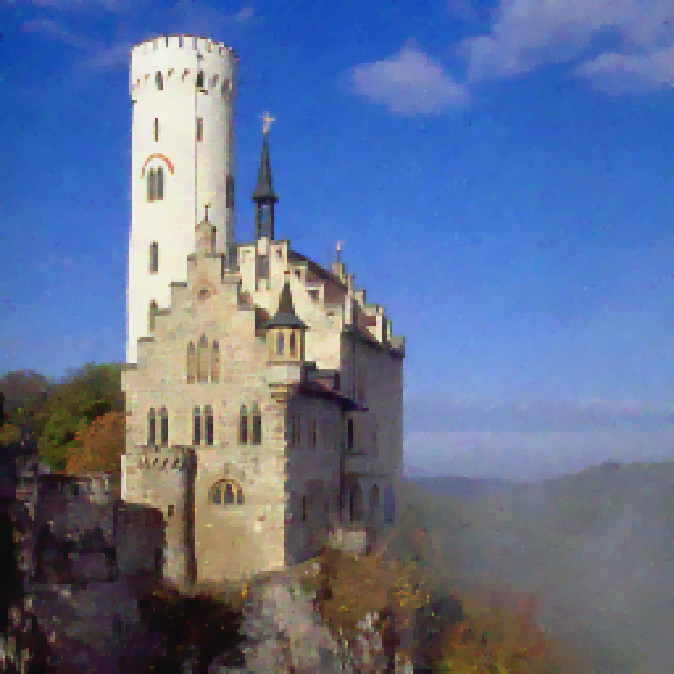}}
\hspace*{0.5cm}
\subfloat[Denoised image, $\lambda_1=0.07$, $\lambda_2=0.01$, isotropic TV]{
\includegraphics[width = 0.31\textwidth]{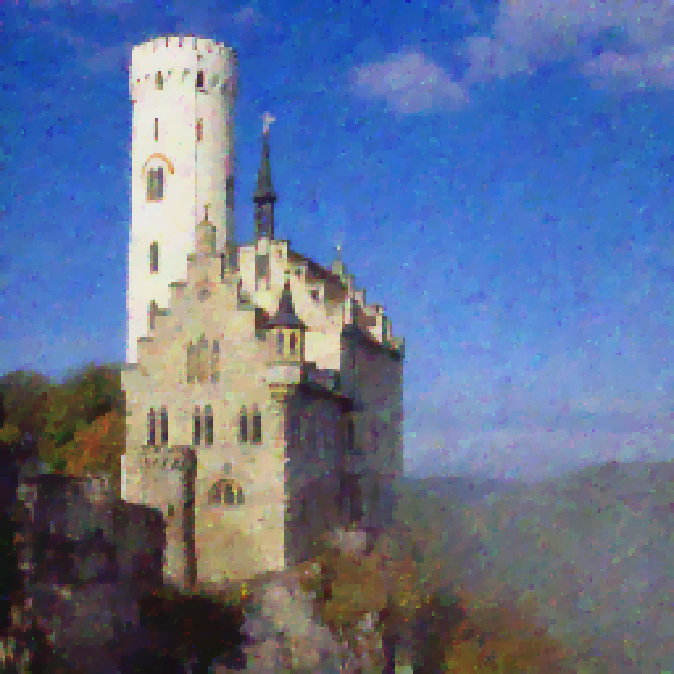}}
\caption{The noisy images in (a) and (b) were obtained after adding white Gaussian noise with standard deviation $\sigma=0.06$ and $\sigma=0.12$, respectively, to the original $256\times 256$ lichtenstein test image. The outputs of Algorithm \ref{alg-f} after $100$ iterations when solving \eqref{probimageproc} with isotropic total variation are shown in (c) and (d), respectively.}
\end{figure}

\begin{table}[ht]\centering
\begin{tabular}{lrrrr}
\toprule
\multirow{2}{*}{$\varepsilon=10^{-4}$}  & \multicolumn{2}{c}{isotropic TV} &  \multicolumn{2}{c}{anisotropic TV} \\
\cmidrule{2-5}
 & $\sigma = 0.06$ & $\sigma = 0.12$ & $\sigma = 0.06$ & $\sigma = 0.12$\\
\midrule
Algorithm in \cite{vu}             & 373 & 329 & 383 & 388\\
Algorithm \ref{alg-f} & 95  & 180 & 126 & 255\\
\bottomrule
\end{tabular}
\caption{Comparison of the algorithm in \cite{vu} and its modification, Algorithm \ref{alg-f}, when denoising the images in Figure 1 (a) and (b) for isotropic TV and anisotropic TV. The entries represent the number of iterations needed for attaining a root mean squared error for the primal iterates below the tolerance level of $\varepsilon=10^{-4}$.}
\end{table}

When considering the \textit{isotropic total variation}, the problem \eqref{probimageproc} can be formulated as
\begin{align}\label{imageiso}
\inf_{x \in \R^k} \left\{f(x) +  g_1(Lx) + g_2(Wx)\right\},
\end{align}
where $f: \R^k \rightarrow \B, f(x) = \frac{1}{2}\|x-b\|^2+\delta_{[0,1]^k}(x)$ is $1$-strongly convex, $ g_1:\Y \rightarrow \R$ is defined as $ g_1(u,v) = \lambda_1 \|(u,v)\|_{\times}$, where $\|(\cdot,\cdot)\|_{\times}:\Y \rightarrow \R$, $\|(u,v)\|_{\times}=\sum_{i=1}^M \sum_{j=1}^N \sqrt{u_{i,j}^2+v_{i,j}^2}$, is a norm on the Hilbert space $\Y$ and $g_2 : \R^k \rightarrow \R, g_2(x) = \lambda_2\|x\|_1$. Take $(p,q)\in \Y$ and $\sigma>0$. We have
$$\prox\nolimits_{\sigma f}(p) = P_{[0,1]^k}\left((1+\sigma)^{-1}(p+\sigma b)\right).$$ Moreover,
$ g_1^*(p,q) = \delta_S(p,q)$ and
$$\prox\nolimits_{\sigma g_1^*}(p,q) =  P_S \left(p,q\right), $$
where (cf. \cite{b-h})
$$ S=\left\{ (p,q)\in\Y : \max_{\substack{1\leq i \leq M\\ 1\leq j \leq N }} \sqrt{p_{i,j}^2+ q_{i,j}^2} \leq \lambda_1 \right\}$$
and the projection operator $P_S : \Y \rightarrow S$ is defined via
\begin{align*}
		(p_{i,j},q_{i,j}) \mapsto \lambda_1 \frac{(p_{i,j},q_{i,j})}{\max\left\{\lambda_1,\sqrt{p_{i,j}^2 + q_{i,j}^2}\right\}}, \ 1\leq i \leq M,\ 1\leq j \leq N.
\end{align*}
Finally, $g_2^*(p)=\delta_{\left[-\lambda_2,\lambda_2\right]^k}(p)$, hence
$$\prox\nolimits_{\sigma g_2^*}(p) = P_{\left[-\lambda_2,\lambda_2\right]^k}(p) \ \forall p \in \R^k.$$

On the other hand, when considering the \textit{anisotropic total variation}, the problem \eqref{probimageproc} can be formulated as
\begin{align}\label{imageaniso}
\inf_{x \in \R^k} \left\{f(x) + \tilde g_1(Lx) + g_2(Wx)\right\},
\end{align}
where the functions $f, g_2$ are taken as above and $\tilde g_1:\Y \rightarrow \R$ is defined as $g_1(u,v)=\lambda_1 \|(u,v)\|_1$. For every $(p,q)\in\Y$ we have $\tilde g_1^*(p) = \delta_{\left[-\lambda_1, \lambda_1\right]^k \times \left[-\lambda_1, \lambda_1\right]^k}(p,q)$ and
$$\prox\nolimits_{\sigma \tilde g_1^*}(p,q) = P_{\left[-\lambda_1, \lambda_1\right]^k \times \left[-\lambda_1, \lambda_1\right]^k}(p,q).$$

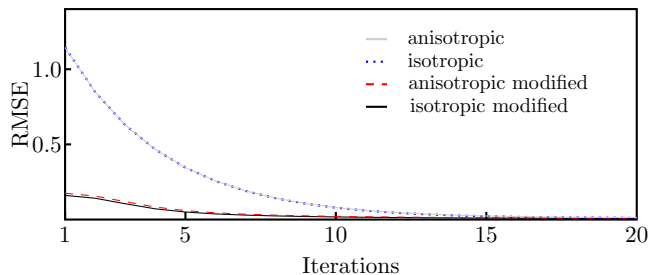
\begin{figure}
\centering
\begin{tikzpicture}[xscale=0.4,yscale=2.0]
\draw[thick] (0,0) rectangle (19,1.4);

\foreach \x in {0,4,9,14}
{ \draw[thick] (\x,0.0) -- (\x,0.05);
}
\draw (0,-0.1) node[scale=0.8]{$1$};
\draw (4,-0.1) node[scale=0.8]{$5$};
\draw (9,-0.1) node[scale=0.8]{$10$};
\draw (14,-0.1) node[scale=0.8]{$15$};
\draw (19,-0.1) node[scale=0.8]{$20$};
\foreach \x in {0.5,1}
{ \draw[thick] (0.0,\x) -- (0.2,\x);
}
\draw (-0.6,0.5) node[scale=0.8]{$0.5$};
\draw (-0.6,1.0) node[scale=0.8]{$1.0$};

\draw[thin] plot file {errorIsoAcc06.dat};
\draw[thin,color=red,dashed] plot file {errorAnisoAcc06.dat};
\draw[thick,color=blue,dotted] plot file {errorIso06.dat};
\draw[thin,color=black!20!white] plot file {errorAniso06.dat};

\draw[thick,color=black!20!white] (10,1.2) -- (10.7,1.2);
\draw (13.0,1.2) node[scale=0.7]{anisotropic};
\draw[thick,color=blue,dotted] (10,1.05) -- (10.7,1.05);
\draw (12.65,1.05) node[scale=0.7]{isotropic};
\draw[thick,color=red,dashed] (10,0.9) -- (10.7,0.9);
\draw (14.4,0.9) node[scale=0.7]{anisotropic modified};
\draw[thick] (10,0.75) -- (10.7,0.75);
\draw (14.1,0.75) node[scale=0.7]{isotropic modified};

\draw (9.5,-0.3) node[scale=0.8]{Iterations};
\draw (-1.5,0.7) node[rotate=90,scale=0.8]{RMSE};
\end{tikzpicture}
\caption{RMSE curves for image denoising: the case of white Gaussian noise with standard deviation $\sigma=0.06$}
\end{figure}

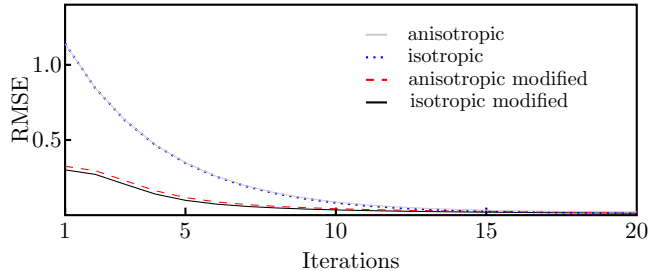
\begin{figure}
\centering
\begin{tikzpicture}[xscale=0.4,yscale=2.0]
\draw[thick] (0,0) rectangle (19,1.4);

\foreach \x in {0,4,9,14}
{ \draw[thick] (\x,0.0) -- (\x,0.05);
}
\draw (0,-0.1) node[scale=0.8]{$1$};
\draw (4,-0.1) node[scale=0.8]{$5$};
\draw (9,-0.1) node[scale=0.8]{$10$};
\draw (14,-0.1) node[scale=0.8]{$15$};
\draw (19,-0.1) node[scale=0.8]{$20$};

\foreach \x in {0.5,1}
{ \draw[thick] (0.0,\x) -- (0.2,\x);
}
\draw (-0.6,0.5) node[scale=0.8]{$0.5$};
\draw (-0.6,1.0) node[scale=0.8]{$1.0$};

\draw[thick,color=black!20!white] (10,1.2) -- (10.7,1.2);
\draw (13.0,1.2) node[scale=0.7]{anisotropic};
\draw[thick,color=blue,dotted] (10,1.05) -- (10.7,1.05);
\draw (12.65,1.05) node[scale=0.7]{isotropic};
\draw[thick,color=red,dashed] (10,0.9) -- (10.7,0.9);
\draw (14.4,0.9) node[scale=0.7]{anisotropic modified};
\draw[thick] (10,0.75) -- (10.7,0.75);
\draw (14.1,0.75) node[scale=0.7]{isotropic modified};

\draw (9.5,-0.3) node[scale=0.8]{Iterations};
\draw (-1.5,0.7) node[rotate=90,scale=0.8]{RMSE};

\draw[thin] plot file {errorIsoAcc12.dat};
\draw[thin,color=red,dashed] plot file {errorAnisoAcc12.dat};
\draw[thick,color=blue,dotted] plot file {errorIso12.dat};
\draw[thin,color=black!20!white] plot file {errorAniso12.dat};
\end{tikzpicture}
\caption{RMSE curves for image denoising: the case of white Gaussian noise with standard deviation $\sigma=0.12$}
\end{figure}

We experimented with the $256 \times 256$ \textit{lichtenstein test image} to which we added white Gaussian noise with standard deviation $\sigma=0.06$ and $\sigma=0.12$, respectively. We solved \eqref{probimageproc} by Algorithm \ref{alg-f} (with the modifications mentioned in Remark \ref{h=0}) for both instances of the discrete total variation functional. For the first noisy image (added noise with standard deviation $\sigma=0.06$), we took as regularization parameters $\lambda_1=0.035$ and $\lambda_2=0.01$ and for the second one (added noise with standard deviation $\sigma=0.12$), $\lambda_1=0.07$ and $\lambda_2=0.01$. As initial choices in Algorithm \ref{alg-f} we opted for $\lambda=1$, $\tau_0 = 50, \sigma_{1,0} = 0.0241$ and $\sigma_{2,0} = 0.008$.  The reconstructed images after $100$ iterations for isotropic total variation are shown in Figure 1.

We compared Algorithm \ref{alg-f} from the point of view of the number of iterations needed for a good recovery with the iterative method from Theorem \ref{vu} (see \cite{vu}) for both isotropic and anisotropic total variation. For the iterative scheme from \cite{vu} we have chosen for both noisy images and both discrete total variation functionals the optimal initializations $\tau = 0.35, \sigma_1=0.2$ and $\sigma_2=0.01$.

The comparisons concerning the number of iterations needed for a good recovery were made via the \textit{root mean squared error (RMSE)}. We refer the reader to Table 1 for the achieved results, which show that Algorithm \ref{alg-f} outperforms the iterative scheme in \cite{vu}. Figures 2 and 3 show the evolution of the RMSE curves when solving \eqref{imageiso} and \eqref{imageaniso} with the algorithm from \cite{vu} and with its modified version, Algorithm \ref{alg-f}, respectively.

\subsection{Support vector machines classifications}\label{subsec4.2}

The numerical experiments we present in this subsection refer to the class of kernel based learning methods and address the problem of classifying images via support vector machines.

We make use of a data set of 11800 training images and 1983 test images of size $28\times 28$ from the website \textit{http://www.cs.nyu.edu/$\sim$roweis/data.html}. The problem consists in determining a decision function based on a pool of handwritten digits showing either the number eight or the number nine, labeled by $-1$ and $+1$, respectively (see Figure 4). We evaluate the quality of the decision function on a test data set by computing the percentage of misclassified images. Notice that we use only a half of the available images from the training data set, in order to reduce the computational effort.

\begin{figure}[ht]
\captionsetup[subfigure]{labelformat=empty}
\centering
\subfloat[]{
\includegraphics[scale = 0.7]{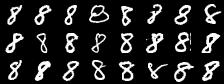}}
\hspace*{0.2cm}
\subfloat[]{
\includegraphics[scale = 0.7]{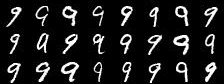}}
\caption{A sample of images belonging to the classes $-1$ and $+1$, respectively.}
\end{figure}

The classifier functional $f$ is assumed to be an element of the \textit{Reproducing Kernel Hilbert Space (RKHS)} ${{\cal H}}_\kappa$, which in our case is induced by the symmetric and finitely positive definite Gaussian kernel function $$\kappa:\R^d\times\R^d\rightarrow \R, \ \kappa(x,y)=\exp \left(-\frac{\|x-y\|^2}{2\sigma^2}\right).$$ Let $\langle\cdot,\cdot\rangle_\kappa$ be the inner product on ${{\cal H}}_\kappa$, $\|\cdot\|_\kappa$ the corresponding norm and $K\in\R^{n\times n}$ the \textit{Gram matrix} with respect to the training data set ${\cal Z}=\{(X_1,Y_1),...,(X_n,Y_n)\}\subseteq \R^d\times \{+1,-1\}$, namely the symmetric and positive definite matrix with entries $K_{ij}=\kappa(X_i,X_j)$ for $i,j=1,...,n$. We make use of the \textit{hinge loss} function $v: \R\times\R\rightarrow\R, v(x,y)=\max\{1-xy,0\}$, which penalizes the deviation between the predicted value $f(x)$ and the true value $y\in\{+1,-1\}$. The smoothness of the decision function $f\in{{\cal H}}_\kappa$ is employed by means of the \textit{smoothness functional} $\Omega:{{\cal H}}_\kappa\rightarrow\R$, $\Omega(f)=\|f\|^2_\kappa$, taking high values for non-smooth functions and low values for smooth ones. The decision function is the optimal solution of the \textit{Tikhonov regularization problem}
\begin{equation}\label{tikh}
\inf_{f\in {{\cal H}}_\kappa}\left\{\frac{1}{2}\Omega(f)+C\sum_{i=1}^n v(f(X_i),Y_i)\right\},
\end{equation}
where $C>0$ denotes the regularization parameter controlling the tradeoff between the loss function and the smoothness functional.

By taking into account the \textit{representer theorem} (see \cite{kernels-book}), there exists a vector $c=(c_1,...,c_n)^T\in\R^n$ such that the minimizer $f$ of \eqref{tikh} can be expressed as a kernel expansion in terms of the training data, namely, it holds $f(\cdot)=\sum_{i=1}^nc_i\kappa(\cdot, X_i)$. In this case the smoothness functional becomes $\Omega(f)=\|f\|_\kappa^2=\langle f,f\rangle_\kappa=\sum_{i=1}^n\sum_{j=1}^n c_ic_j\kappa(X_i,X_j)=c^TKc$ and for $i=1,...,n$ it holds $f(X_i)=\sum_{j=1}^nc_j\kappa(X_i,X_j)=(Kc)_i$. This means that in order to find the decision function it is enough to solve the convex optimization problem
\begin{equation}\label{supp}
\inf_{c\in\R^n}\left\{h(c)+\sum_{i=1}^n g_i(Kc)\right\},
\end{equation}
where $h:\R^n\rightarrow\R$, $h(c)=\frac{1}{2}c^TKc$ is convex, differentiable with $\|K\|$-Lipschitz gradient and $g_i:\R^n\rightarrow\R$, $g_i(c)=Cv(c_i,Y_i)$, $i=1,...,n,$ are convex functions. The optimization problem \eqref{supp} has the structure of \eqref{sum-k-prim-f}, where $f$ is taken to be identically $0$. Notice that in this case the function $f+h=h$ is $\gamma$-strongly convex with $\gamma=\lambda_{\min}$, where $\lambda_{\min}$ is the smallest eigenvalue of the matrix $K$. Due to the continuity of the functions $g_i, i=1,...,n$, the qualification condition required in Theorem \ref{th-f} is guaranteed. We solved \eqref{supp} by Algorithm \ref{alg-f} and used for $\mu > 0$ the following formula (see \cite{b-hb-var})
$$\prox\nolimits_{\mu g_i^*}(c)=\left(0,...,P_{Y_i[-C,0]}(c_i-\mu Y_i),...,0\right)^T.$$
With respect to the considered data set, we denote by $\cal{D}$$=\{(X_i,Y_i),i=1,...,5899\}\subseteq \R^{784}\times\{+1,-1\}$ the set of available training data consisting of $2974$ images in the class $-1$ and $2925$ images in the class $+1$. A sample from each class of images is shown in Figure 4. The images have been vectorized and normalized by dividing each of them by the quantity $\left(\frac{1}{5899}\sum_{i=1}^{5899}\|X_i\|^2\right)^{\frac{1}{2}}$.

As initial choices in Algorithm \ref{alg-f} we took $\tau_0=0.99\frac{2\gamma}{\|K\|}, \lambda=\|K\|+1$ and $\sigma_{i,0}=\sqrt{1+\tau_0(2\gamma-\|K\|\tau_0)/\lambda}/(n\tau_0\|K\|^2)$, $i=1,...,n$, and tested different combinations of the kernel parameter $\sigma$ over a fixed number of $1500$ iterations. In Table 2 we present the misclassification rate in percentage for the training and for the test data (the error for the training data is less than the one for the test data). One can notice that for certain choices of $\sigma$ the misclassification rate outperforms the one reported in the literature dealing with numerical methods for support vector classification. Let us mention that the numerical results are given for the case $C=1$. We tested also other choices for $C$, however we did not observe great impact on the results.

\begin{table}\centering
\begin{tabular}{rrrrrr}
\toprule
%  &  $\sigma=0.15$ & $\sigma=0.175$ & $\sigma=0.2$ & $\sigma=0.25$ & $\sigma=0.5$\\
$\sigma$ &  $0.15$ & $0.175$ & $0.2$ & $0.25$ & $0.5$\\
\midrule
training error & $0$ & $0$ & $0.14$ & $4.09$ & $49.55$\\
test error     & $1.36$ & $2.12$ & $3.33$ & $5.60$ & $49.12$ \\
\bottomrule
\end{tabular}
\caption{Misclassification rate in percentage for different choices of the kernel parameter $\sigma$ and for both the training and the test data set.}
\end{table}

\end{document}